\documentclass{amsart}

\usepackage[T1]{fontenc}
\usepackage[utf8]{inputenc}

\usepackage{amssymb,amsmath,amsfonts,amsthm,mathrsfs,mathtools,amscd,tikz-cd,dirtytalk,comment}
\usepackage[hidelinks]{hyperref}
\usepackage[capitalise]{cleveref}

\newtheorem{theorem}{Theorem}[section]
\newtheorem{lemma}[theorem]{Lemma}
\newtheorem{proposition}[theorem]{Proposition}
\newtheorem{corollary}[theorem]{Corollary}

\theoremstyle{definition}
\newtheorem{definition}[theorem]{Definition}
\newtheorem{example}[theorem]{Example}

\theoremstyle{remark}
\newtheorem{remark}[theorem]{Remark}

\numberwithin{equation}{section}

\DeclareMathOperator{\Aut}{Aut}

\DeclareMathOperator{\id}{id}

\DeclareMathOperator{\Inj}{Inj}

\DeclareMathOperator{\Ret}{Ret}
\DeclareMathOperator{\Soc}{Soc}
\DeclareMathOperator{\Perm}{Perm}

\DeclareMathOperator{\Triv}{Triv}

\DeclareMathOperator{\mpl}{mpl}
\DeclareMathOperator{\Stab}{Stab}

\newcommand{\N}{\mathbb{N}}
\newcommand{\Z}{\mathbb{Z}}
\newcommand{\Q}{\mathbb{Q}}

\newcommand{\cO}{\mathcal{O}}
\newcommand{\M}{\mathcal{M}}

\newcommand{\G}{\mathcal{G}}

\newcommand{\AbBr}{\mathbf{AbBr}}
\newcommand{\AbSol}{\mathbf{AbSol}}

\begin{document}

\title[Studying solutions of the YBE through skew braces]{Studying solutions to the Yang--Baxter equation through skew braces, with an application to indecomposable involutive solutions with abelian permutation group}

\author{M.~Castelli}

\address{Lecce, Italy}         
\email{marco.castelli@unisalento.it - marcolmc88@gmail.com}
\urladdr{https://www.researchgate.net/profile/Marco-Castelli-3}

\thanks{The author is a member of the GNSAGA (INdAM)}

\author{S.~Trappeniers}

\address{Department of Mathematics and Data Science, 
 Vrije Universiteit Brussel, 
 Pleinlaan 2, 
 1050 Brussels, Belgium} 

\email{senne.trappeniers@vub.be} 
\urladdr{https://sites.google.com/view/sennetrappeniers/}

\thanks{The author was supported by Fonds Wetenschappelijk Onderzoek – Vlaanderen, grant 1160522N}



\begin{abstract}
We connect properties of set-theoretic solutions to the Yang--Baxter equation to properties of their permutation skew brace. In particular, a variation of the multipermutation level of a solution is presented and we show that it coincides with the multipermutation level of the permutation skew brace, contrary to the inequality that one has for the usual multipermutation level of solutions. We relate the number of orbits of a solution to generators of its permutation skew brace and relate different kinds of notions of generating sets of a skew brace. Also, the automorphism groups of solutions are studied through their permutation skew brace. As an application, we obtain a surprising result on subsolutions of multipermutation solutions and we give a description of all finite indecomposable involutive solutions to the Yang--Baxter equation with abelian permutation group. For multipermutation level 3, we obtain the precise number of isomorphism classes of such solutions of a given size.
\end{abstract}

\maketitle

\section{Introduction}

Motivated by the study of set-theoretic solutions to the Yang-Baxter equation, suggested in 1992 by Drinfel'd \cite{drinfeld1992some}, a considerable part of the literature has been devoted to the study of skew braces, algebraic structures introduced by Rump in \cite{rump2007braces} and generalised in a non-commutative setting by Guarnieri and Vendramin in \cite{GV17}. As highlighted in the theory developed in \cite{Ba18,bachiller2016solutions,rump2007braces}, the classification of skew braces is a milestone for the classification of non-degenerate set-theoretic solutions to the Yang-Baxter equation (which we will simply call solutions); indeed, skew braces provide solutions and, conversely, every solution arises from a skew brace. Even before the introduction of skew braces, several authors focused on the study of solutions to the Yang-Baxter equation. The seminal papers of Gateva-Ivanova and Van den Bergh \cite{gateva1998semigroups} and Etingof, Schedler and Soloviev \cite{etingof1998set} attracted attention on the involutive solutions. In this direction, a successful approach comes from \cite[Section 2]{etingof1998set}, in which the authors developed an extension-tool that allows to construct every involutive non-degenerate solution starting from the so called indecomposable solutions. In light of this fact, the classification of indecomposable involutive solutions is the first step in the classification-problem of involutive solutions. 
In the last years, several authors developed different techniques to study indecomposable solutions.
In \cite[Theorem 2.13]{etingof1998set} all the indecomposable involutive solutions with a prime number of elements were classified using bijective $1$-cocycles. In \cite{cacsp2018}, the first author, with Catino and Pinto, developed a cycle set theory for indecomposable involutive solutions. Also, Jedlička, Pilitowska and Zamojska-Dzienio used $2$-reductive left-quasigroups (see also \cite{JePiZa}) to describe all the indecomposable involutive solutions with multipermutation level at most $2$ \cite{JePiZa22x} and, among these ones, to classify those having arbitrary cardinality and abelian permutation group \cite{JPZ2021}.
However, the theory of braces allows to obtain further strong results in this field. Indeed, in \cite{cedo2022indecomposable} Ced\'o and Okni\'nski described indecomposable involutive solutions having square-free size by their permutation brace, while in \cite{cedo2020primitive}, together with Jespers, they classified involutive non-degenerate solutions having primitive permutation group. In \cite{rump2020,smock}, the authors prove several results relating indecomposable involutive multipermutation solutions and one-generated braces. Using this theory, indecomposable involutive solutions having cyclic permutation group were completely classified in \cite{jedlivcka2021cocyclic}. Braces turned out to be particularly useful to understand a particular class of indecomposable solutions, the so-called simple solutions: indeed, these solutions were recently studied by means of braces in \cite{cedo2021constructing,cedo2022new} and completely described in \cite{castelli2022characterization}. For further recent results on indecomposable involutive solutions, we refer the reader to \cite{ramirez2022decomposition,ramirez2022indecomposable,dietzel2023indecomposable,lebed2022involutive}. Some authors recently started the study of non-involutive indecomposable solutions by means of tools that are suitable generalizations of the ones used in the involutive case; for some references, see \cite{EtSoGu01,castelli2022simplicity,Ga21,colazzo2022derived,CSV19,CJKVAV23}.

The results in this paper are driven by two questions, the first one being one of the central questions in most current research on skew braces and the Yang--Baxter equation.
\begin{enumerate}
    \item What information from a solution can be recovered from its permutation skew brace and vice versa?
    \item Can we classify indecomposable involutive solutions with abelian permutation group?
\end{enumerate}
After covering important definitions, constructions and results in \cref{section: preliminaries}, we obtain some first new results in \cref{section: variation mpl}. Here we introduce a slight variation of the multipermutation level of a solution. For the main result of this section, \cref{theorem: mpl of permutation skew brace is mpl of solution}, we show that this variation can be related in a exact way with the multipermutation level of the permutation skew brace, contrary to the usual multipermutation level. In \cref{section: generators of skew brace} we relate the existence of transitive cycle bases of a skew brace, which are inherently connected to indecomposable solutions, to results concerning the generators of the skew brace. In this way, we are able to extend results of \cite{smock,rump2020one} to the non-involutive case and also link the existence of cycle bases to the weight of a skew brace, as introduced in \cite{JKVAV2021}. As a surprising application of the results on generators of skew braces, in \cref{section: an application to mp solutions} we obtain that an indecomposable solution of finite multipermutation level has no subsolutions. This implies that all endomorphisms of a finite indecomposable solution of finite multipermutation level are bijective, providing a strong generalisation of \cite[Proposition 5.1]{JPZ2021}. In \cref{section: automorphism group}, we look at the automorphism group of indecomposable involutive solutions with a regular permutation group. We obtain obtain both results on the order and the structure of the automorphism group, extending some results contained in \cite{JPZ2021}. For those indecomposable involutive solutions with an abelian permutation group and multipermutation level 2, we prove that their automorphism group is isomorphic to the additive group of the permutation brace. In \cref{section: finite indecomposable solutions}, we then turn to the second question mentioned above. Here we associate to certain matrices a finite indecomposable involutive solution with abelian permutation group and we show that every such solution can be obtained in this way. However, solutions coming from different matrices might give isomorphic solution. In the case where the multipermutation level is 2 or 3, we are able to take into account these isomorphisms and count explicitly the number of isomorphism classes, extending the enumeration and classification results in \cite{JPZ2021}. At last, in \cref{section: infinite solutions} we use some of the techniques of the preceding section to study also infinite indecomposable involutive solution with abelian permutation group and finite multipermutation level. We prove that every infinite indecomposable involutive solution with abelian permutation group and multipermutation level 2 is isomorphic to one of the examples constructed in \cite{JPZ2021} and obtain results on infinite indecomposable involutive solution with torsion-free  abelian permutation group.

\section{Preliminaries}\label{section: preliminaries}
A \emph{skew (left) brace} is a triple $(A,+,\circ)$ where $A$ is a set, $+$ and $\circ$ are binary operations such that $(A,+)$ and $(A,\circ)$ are groups, and the equality 
\begin{equation}
    a\circ(b+c)=a\circ b-a+a\circ c,\label{eq: left brace}
\end{equation}
holds for all $a,b,c\in A$. Here, $-a$ denotes the inverse of $a$ in the group $(A,+)$ and $\circ$ has precedence over $+$. The inverse of an element $a\in A$ with respect to the group $(A,\circ)$ is denoted by $a^-$. The group $(A,+)$ is the \emph{additive group} of A and $(A,\circ)$ is the \emph{multiplicative group} of $A$. It follows directly from \eqref{eq: left brace} that the identity elements of $(A,+)$ and $(A,\circ)$ coincide, this element is denoted by $0$. For skew braces $A$ and $B$, a map $f:A\to B$ is a \emph{skew brace homomorphism} if $f(a+b)=f(a)+f(b)$ and $f(a\circ b)=f(a)\circ f(b)$ for all $a,b\in A$. The group of skew brace automorphisms of $A$ is denoted by $\Aut(A,+,\circ)$. A skew brace with an abelian additive group is called a \emph{brace}. If for a skew brace $A$, also 
\begin{equation*}
    (a+b)\circ c=a\circ c-c+b\circ c,
\end{equation*}
for all $a,b,c\in A$, then we say that $A$ is a \textit{two-sided skew brace}. Note that any skew brace with an abelian multiplicative group is in particular two-sided.
For any group $(A,\circ)$ we can make $A$ into a two-sided skew brace $(A,\circ,\circ)$. Such a skew brace, where the two operations coincide, is called \textit{trivial}.

For any $a\in A$, the map 
\begin{equation*}
    \lambda_a:A\to A,\quad b\mapsto -a+a\circ b,
\end{equation*}
is an automorphism of $(A,+)$ and moreover we obtain a group homomorphism 
\begin{equation*}
    \lambda:(A,\circ)\to \Aut(A,+),\quad a\mapsto \lambda_a,
\end{equation*}
which is called the \emph{$\lambda$-map} of $A$ \cite{GV17}.
If also $\lambda:(A,+)\to \Aut(A,+)$ is a group homomorphism, then we say that $A$ is $\lambda$-homomorphic. We say that a skew brace $(A,+,\circ)$ is a \textit{bi-skew} brace if also $(A,\circ,+)$ is a skew brace. Bi-skew braces are precisely those skew braces such that $\lambda_a$ is a skew brace automorphism for all $a\in A$.

A subgroup $I$ of $(A,+)$ such that $\lambda_a(I)\subseteq I$ for all $a\in A$, is a \emph{left ideal}. A \textit{strong left ideal} is a left ideal which is moreover a normal subgroup of $(A,+)$.
An \emph{ideal} of $A$ is a strong left ideal that moreover is a normal subgroup of $(A,\circ)$. For an ideal $I$, the \emph{quotient skew brace} $A/I$ can be defined in the natural way. An important example of an ideal of a skew brace $A$ is the socle, 
$$\Soc(A)=\{x\in A\mid x\circ a=x+a=a+x\ \text{for all $a\in A$}\}.$$
If $A$ is a brace, then $\Soc(A)=\ker \lambda$. More generally, we can define the socle series of a skew brace $A$ by $\Soc_0(A)=0$ and $\Soc_{n+1}(A)$ is the ideal containing $\Soc_n(A)$ such that $\Soc_{n+1}(A)/\Soc_n(A)=\Soc(A/\Soc_n(A))$ for $n\geq 0$. A skew brace $A$ is said to be \textit{multipermutation} skew brace if there exists some $n$ such that $\Soc_n(A)=A$. In \cite{CJKVAV23} the smallest such $n$, if it exists, is called the \textit{socle length} of $A$. We will however call this the \textit{multipermutation level} of $A$ and denote it by $\mpl(A)$. The reason is the following, and will be more clear after we introduce the multipermutation level of a solution: we can define the \textit{retraction} of a skew brace as $\Ret(A)=A/\Soc(A)$. By induction we let $\Ret^0(A)=A$ and $\Ret^{n+1}(A)=\Ret(\Ret^{n}(A))$ for $n\in \N$. Then clearly $\Ret^n(A)\cong A/\Soc_n(A)$ and thus $\Ret^n(A)=0$ if and only if $\Soc_n(A)=A$.

It is easily verified that $\theta_{(a,b)}(x)=a+\lambda_b(x)-a$ yields a well-defined action $\theta:(A,+)\rtimes (A,\circ)\to \Aut(A,+)$. In the semidirect product, $(A,\circ)$ acts on $(A,+)$ by $\lambda$. Then we can see strong left ideals of $A$ as the subgroups of $(A,+)$ mapped to themselves under the action of $\theta$. Also, the socle of $A$ contains precisely those elements $x\in A$ such that $(x,0),(0,x)\in \ker \theta$.

On any skew brace $A$ we define a new operation $*$ as $a*b=-a+a\circ b-b$, which can be seen as a measure of the difference between $a\circ b$ and $a+b$. If $A$ is a two-sided brace, then $(A,+,*)$ is a Jacobson radical ring, which by definition is a ring which coincides with its own Jacobson radical. Conversely, every Jacobson radical ring $(A,+,*)$ gives a two-sided brace $(A,+,\circ)$ where $a\circ b=a+a*b+b$ \cite{rump2007braces}. The ideals of a two-sided brace coincide with the ideals of its associated Jacobson radical ring. Also for a general skew brace $A$, this operation $*$ is useful: for $X,Y\subseteq A$ we define $X*Y$ as the additive subgroup of $A$ generated by $\{x*y\mid x\in X,y\in Y\}$. Following \cite{rump2007braces,Smo18}, we introduce two descending series of subskew braces. We set $A^{(1)}=A^1=A$. The left series of a skew brace $A$ is given by $A^{n+1}=A*A^n$, where $n\geq 1$. $A$ is said to be \emph{left nilpotent} if $A^{n }=0$ for some $n$. Similarly, the right series of a skew brace $A$ is defined as $A^{(n+1)}=A^{(n)}*A$, where $n\geq 1$. All $A^{(n)}$, so in particular $A^{(2)}=A^2$, are ideals. $A$ is said to be \emph{right nilpotent} if $A^{(n)}=0$ for some $n$. Similarly, note that for two-sided braces, left and right nilpotency coincide with nilpotency of the associated Jacobson radical ring. In \cite{CSV19} it is proved that a skew brace is multipermutation if and only if it is right nilpotent and its additive group is nilpotent.  By results in \cite{Car20,ST23} we have the following proposition.

\begin{proposition}
    For a brace $A$, the following are equivalent:
\begin{enumerate}
\item $A$ is a bi-skew brace,
\item $A$ is $\lambda$-homomorphic,
\item $\mpl(A)\leq 2$,
\item $A^{(3)}=0$,
\item $\lambda_a$ is a skew brace automorphism for each $a\in A$.
\end{enumerate}
\end{proposition}
We say that a skew brace is \textit{annihilator nilpotent} if it is left nilpotent and multipermutation. This is not the original definition as in \cite{JVAV22}, but coincides with their definition because of \cite[Corollary 2.15]{JVAV22}.

A \emph{set-theoretic solution to the Yang--Baxter equation} is a pair $(X,r)$ with $X$ a nonempty set and $r\colon X^2\to X^2$ a bijective map satisfying the braid equation
\begin{equation*}
    r_1r_2r_1=r_2r_1r_2,
\end{equation*}
where $r_1=r\times \id_X$ and $r_2=\id_X\times r$. On every nonempty set $X$, the map $(x,y)\mapsto (y,x)$ satisfies the braid equation, such a solution is said to be \emph{trivial}. A solution is \emph{non-degenerate} if the maps $\sigma_x,\tau_x\colon X\to X$ defined by $r(x,y)=(\sigma_x(y),\tau_y(x))$ are bijective. In the rest of this section, we will shortly refer to a non-degenerate set-theoretic solution to the Yang--Baxter equation as a \emph{solution}. Furthermore, we say that a solution is \emph{involutive} if $r^2=\id$. For an involutive solution $(X,r)$, the maps $\tau_x$ are completely determined by the maps $\sigma_x$ as $\tau_x(y)=\sigma^{-1}_{\sigma_x(y)}(x)$ for all $x,y\in X$. It therefore suffices to give only the maps $\sigma_x$ when describing an involutive solution.

Given two solutions $(X,r)$ and $(Y,s)$, we say that a map $f\colon X\to Y$ is a \emph{homomorphism of solutions} if $(f\times f)r=s(f\times f)$. One can prove that in this case the image of $f$ is a subsolution of $(Y,s)$, meaning that $s$ restricts to $f(X)\times f(X)$. A bijective homomorphism of solutions is called an \emph{isomorphism of solutions}.

Given a skew brace $A$, define 
\begin{equation*}
    r_A\colon A^2\to A^2, \quad (a,b)\mapsto (\lambda_a(b),\lambda^{-1}_{\lambda_a(b)}(-(a\circ b)+a+(a\circ b))).
\end{equation*}
Then the pair $(A,r_A)$ is a solution~\cite{LYZ00,GV17}. This construction is functorial: a skew brace homomorphism $f\colon A\to B$ induces a homomorphism of solutions $f\colon (A,r_A)\to (B,r_B)$. 

Following~\cite{etingof1998set} we define the \emph{structure group} of a solution $(X,r)$ as 
\begin{equation*}
    G(X,r)=\langle X\mid x\circ y=\sigma_x(y)\circ \tau_y(x) \text{ for all }x,y\in X\rangle.
\end{equation*}
For each solution $(X,r)$, one can also define its \emph{derived structure group} as
\begin{equation*}
    A(X,r)=\langle X\mid x+ y=y+ \sigma_y\tau_{\sigma^{-1}_x(y)}(x)\rangle.
\end{equation*}
It is possible to construct a bijection between $A(X,r)$ and $G(X,r)$ such that, transferring the group structure of $A(X,r)$ to $G(X,r)$, one obtains a skew brace $(G(X,r),+,\circ)$. This is a brace if and only if $(X,r)$ is involutive; see~\cite{Sol00,LV19}. Under this bijection, the generator of $A(X,r)$ corresponding to some element $x\in X$ is mapped to the generator of $G(X,r)$ corresponding to $x$. We define $\iota\colon X\to G(X,r)$ as the canonical map sending each element to its corresponding generator. We therefore find that $\iota(X)\subseteq (G(X,r),+,\circ)$ generates both the additive and multiplicative group. Moreover, the skew brace $(G(X,r),+,\circ)$ satisfies the property that $\iota$ is a homomorphism of solutions $\iota\colon (X,r)\to (G(X,r),r_{G(X,r)})$. In particular, this means that for all $x,y\in X$ the equality $\iota(\sigma_x(y))=\lambda_{\iota(x)}{(\iota(y))}$ holds in $G(X,r)$. A solution is \emph{injective} if $\iota$ is an injective map and one can show that involutive solutions are always injective. The map $\iota$ also satisfies the following universal property: if $A$ is any skew brace and $f:(X,r)\to (A,r_A)$ is a homomorphism of solutions, then there exists a unique skew brace homomorphism $\overline{f}:G(X,r)\to A$ such that $f=\overline{f}\pi$. From the universal property, we see that injective solutions are those that appear as a subsolution of $(A,r_A)$ for some skew brace $A$. To any solution $(X,r)$ one can associate its \emph{injectivization} $\Inj(X,r)$, which is the image of the homomorphism $\iota\colon(X,r)\to (G(X,r),r_{G(X,r)})$. It is clear that $G(\Inj(X,r))$ and $G(X,r)$ are isomorphic as skew braces and therefore $\Inj(X,r)$ is indeed an injective solution. 

The group $(G(X,r),\circ)$ acts on $X$, where $\iota(x)$ maps $y$ to $\sigma_x(y)$ for all $x,y\in X$. Also the group $(G(X,r),+)$ acts on $X$, where $\iota(x)$ maps $\sigma_x\tau_{\sigma^{-1}_y(x)}(y)$ to $y$. These actions are compatible in the sense that we can combine them into an action of $(G(X,r),+)\rtimes (G(X,r),\circ)$ on $X$, where $(G(X,r),\circ)$ acts on $(G(X,r),+)$ by $\lambda$. Note that if $(X,r)$ is injective and we identify $\iota(X)$ with $X$, then the action $\theta$ of $(G(X,r),+)\rtimes (G(X,r),\circ)$ on $G(X,r)$ restricts to $\iota(X)$ and this restriction coincides with the above action. Orbits of the above action are simply referred to as the \textit{orbits of} $(X,r)$. A solution $(X,r)$ is \textit{indecomposable} if $X$ consists of a single orbit. This coincides with the definition of indecomposable solutions in \cite{EtSoGu01}, although we work with a slightly different action. A possibly more intuitive way to view orbits of a solutions $(X,r)$ is as the smallest partition of $X$ which is preserved under all maps $\sigma_x,\sigma_x^{-1},\tau_x,\tau^{-1}_x$ for $x\in X$.

Following~\cite{Sol00}, we can also associate to any solution $(X,r)$ the \emph{permutation group}, defined as 
\begin{equation*}
    \mathcal{G}(X,r)=\langle (\sigma_x,\tau_x^{-1})\mid x\in X\rangle \subseteq \Perm(X)\times \Perm(X).
\end{equation*}
There exists a unique surjective group homomorphism $\pi\colon G(X,r)\to \mathcal{G}(X,r)$ satisfying $\pi(\iota(x))=(\sigma_x,\tau_x^{-1})$. It can be shown that the kernel $\Gamma$ of this map $\pi$ is an ideal of $(G(X,r),+,\circ)$, hence there is a natural skew brace structure $(\mathcal{G}(X,r),+,\circ)$ such that $\pi$ is a skew brace homomorphism. We call this the \textit{permutation skew brace}. The ideal $\Gamma$ is contained in $\Soc(G(X,r))$ and if $(X,r)$ is injective, then $\Gamma=\Soc(G(X,r))$. In particular, note that for any involutive (hence injective) solution $(X,r)$, we find that $\Gamma=\ker (\lambda)$. For an involutive solution $(X,r)$, we say that it has a regular permutation group if the action of $(\G(X,r),\circ)$ on $X$ (coming from the action of $(G(X,r),\circ)$ on $X$) is regular. As the construction of the solution associated to a skew brace is functorial, we find a homomorphism of solutions $\pi\iota\colon (X,r)\to (\G(X,r),r_{\G(X,r)})$. The image of this homomorphism is called the \emph{retraction} of $(X,r)$, denoted by $\Ret(X,r)$. The retraction $\Ret(X,r)$ can also be obtained as the induced solution on the equivalence classes of the equivalence relation given by
\begin{equation*}
    x\sim y\iff \sigma_x=\sigma_y\text{ and }\tau_x=\tau_y,
\end{equation*}
which is its original definition in literature; see \cite{etingof1998set,CJKVAV23}. We inductively define $\Ret^0(X,r)=(X,r)$ and $\Ret^{n+1}(X,r)=\Ret(\Ret^n(X,r))$ for $n\geq 0$. If there exists some $n$ such that $\Ret^n(X,r)$ has size $1$, we say that $(X,r)$ is a \textit{multipermutation solution}. In this case, the \textit{multipermutation level} of $(X,r)$, denoted $\mpl(X,r)$, is the smallest $n$ such that $\Ret^n(X,r)$ has size $1$. A solution of multipermutation level 1 is called a \textit{permutation solution}. Equivalently, a solution $(X,r)$ is a permutation solution if $\sigma_x=\sigma_y$ and $\tau_x=\tau_y$ for all $x,y\in X$. As noted in \cite[Remark 4.4]{CJKVAV23}, for any skew brace $A$ we have that $\Ret_n(A,r_A)=(\Ret_n(A),r_{\Ret_n(A)})$ so a skew brace $A$ is multipermutation precisely if $(A,r_A)$ is a multipermutation solution and $\mpl(A)=\mpl(A,r_A)$.

Let $A$ be a skew brace. A nonempty set $X\subseteq A$ is a \textit{cycle base} if it is a union of orbits under the action $\theta$ and $X$ generates the additive (or equivalently the multiplicative) group of $A$. If $X$ consists of a single orbit, then $X$ is a \textit{transitive cycle base}. Cycle bases for braces coincide with the common definition in literature. For skew braces this is a new notion, whose definition already showed up in literature (e.g. \cite{Ba18}) but without being given an explicit name. If $(X,r)$ is a solution then $\iota(X)$ is a cycle base of $G(X,r)$, hence also the image of the canonical map $(X,r)\to \G(X,r)$ is a cycle base.

 Bachiller showed in \cite{Ba18} how starting from a cycle base $X$ of a skew brace $A$, one can construct a solution such that its permutation skew brace brace is precisely $A$ and here proved that every solution can be constructed in this way from its permutation skew brace. Under this construction, an indecomposable solution can only be obtained from a transitive cycle base and every transitive cycle base yields an indecomposable solution. We mention the following more restrictive versions of \cite[Theorems 3.19 and 3.20]{Ba18} (which also follow from \cite[Theorems 3.1 and 4.1]{bachiller2016solutions} or the results in \cite{rump2019covering, rump2022class}).
\begin{theorem}\label{theorem: construction solution Bachiller}
    Let $A$ be a brace. Then for any transitive cycle base $X$ of $A$ and $x\in X$, $(A,r_x)$ with
    \[r_x(a,b)=(\sigma_a(b),\sigma^{-1}_{\sigma_a(b)}(a)),\]
    where $\sigma_a(b)=\lambda_a(x)\circ b$, is an indecomposable solution with regular permutation group such that $\G(A,r)\cong A$ as braces. Furthermore, any indecomposable solution $(Y,s)$ with a regular permutation group and such that $\G(Y,s)\cong A$, is isomorphic to such a solution.
\end{theorem}
Given an element $z\in A$, then we denote the right translation by $z$ in $(A,\circ)$ by $t_z$.
\begin{theorem}\label{theorem: isomorphism solution Bachiller}
    Let $A$ be a left brace and let $(A,r_x)$ and $(A,r_{y})$ be solutions as obtained in \cref{theorem: construction solution Bachiller}. Given $z\in A$ and automorphism $\psi$ of $A$ such that $\psi(x)=\lambda_z(y)$, then 
    \[t_z\circ \psi:(A,r_x)\to(A,r_{y}),\quad a\mapsto \psi(a)\circ z\] 
    is an isomorphism of solutions and moreover every isomorphism between $(A,r_x)$ and $(A,r_{y})$ arises in this way.
\end{theorem}
If $X$ is a union of orbits of a skew brace $A$, then the solution $(A,r_A)$ can be restricted to a solution $(X,r_A|_{X\times X})$. If moreover $X$ is a cycle base, then the orbits of $(X,r_A|_{X\times X})$ coincide with the orbits of $X$ in $A$, because $X$ generates both $(A,+)$ and $(A,\circ)$.

\section{A variation of the multipermutation level}\label{section: variation mpl}
In this section, we introduce a variation of the multipermutation level of a solution and show that this is closer related to the multipermutation level of the associated skew braces than the usual multipermutation level.
\begin{definition}
Let $(X,r)$ be a multipermutation solution. We define $\mpl'(X,r)$ as the smallest $n$ such that $\Ret^{n}(X,r)$ is a trivial solution (possibly of size $>1$). 
\end{definition}
If $(X,r)$ is a multipermutation solution of level $n$, then $\Ret^{n}(X,r)$ has size 1 and is therefore trivial. Also note that if $(X,r)$ is a solution and $\Ret^n(X,r)$ is a trivial solution, then $(X,r)$ is a multipermutation solution as $|\Ret^{n+1}(X,r)|=1$. This proves the following result.
\begin{lemma}\label{lem: trivial inequality mpl}
    For $(X,r)$ a multipermutation solution, we have
    \begin{equation}\label{eq: inequality mpl mpl'}
    \mpl'(X,r)\leq \mpl(X,r)\leq \mpl'(X,r)+1.
\end{equation}
\end{lemma}

\begin{proposition}\label{prop: mpl indecomposable}
    Let $(X,r)$ be an indecomposable solution of finite multipermutation level, then $\mpl'(X,r)=\mpl(X,r)$.
\end{proposition}
\begin{proof}
    Let $n$ be such that $\Ret^n(X,r)$ is a trivial solution. As indecomposability is preserved under retraction, $\Ret^n(X,r)$ is also indecomposable hence $|\Ret^n(X,r)|=1$. Therefore the statement follows.
\end{proof}
As a generalisation of \cite[Definition 4.3]{gateva2018set}, we introduce the following condition.
\begin{definition}
    Let $(X,r)$ be a solution. We say that $(X,r)$ satisfies condition $(*)$ if for every $x\in X$, there exist $y,y'\in X$ such that $\sigma_y(x)= x$ and $\tau_{y'}(x)=x$. 
\end{definition}
It is easily seen that solutions $(X,r)$ such that $r(x,x)=(x,x)$ for all $x\in X$, also known as \textit{square-free solutions}, satisfy condition $(*)$. Also, the solution $(A,r_A)$ associated to a skew brace $A$ satisfies condition $(*)$ because $\sigma_0=\tau_0=\id$.
\begin{proposition}\label{prop: mpl square-free}
Let $(X,r)$ be a multipermutation solution with $|X|>1$ and satisfying property $(*)$. Then $\mpl(X,r)=\mpl'(X,r)+1$.
\end{proposition}
\begin{proof}
Let $n=\mpl(X,r)$. As $|X|>1$, we know that $n\geq 1$, hence $\Ret^{n-1}(X,r)$ is well-defined and must be a permutation solution on a set of size more than 1. The condition $(*)$ is preserved under retractions and a permutation solution satisfying this condition is necessarily trivial. It follows that $\Ret^{n-1}(X,r)$ is a trivial solution and $\mpl'(X,r)\leq n-1$ and therefore the statement follows in combination with \cref{lem: trivial inequality mpl}.
\end{proof}
The following lemma appears as Proposition~4.8 in \cite{CJKVAV23} except for the inequality about $\mpl'(X,r)$, but the proof is the same as for the one about $\mpl(X,r)$.
\begin{lemma}\label{lem: multipermutation surjection}
Let $f:(X,r)\to (Y,s)$ be a surjective homomorphism of solutions, then $f$ induces a surjective homomorhism of solutions $\overline{f}:\Ret(X,r)\to \Ret(Y,s)$. 
If $(X,r)$ is a multipermutation solution, then so is $(Y,s)$ and moreover $\mpl(Y,s)\leq \mpl(X,r)$ and $\mpl'(Y,s)\leq \mpl'(X,r)$.
\end{lemma}
\begin{corollary}\label{cor: multipermutation level of injectivization}
Let $(X,r)$ be a multipermutation solution, then 
\begin{align*}
    \mpl(X,r)-1&\leq \mpl\Inj(X,r)\leq \mpl(X,r),\\
    \mpl'(X,r)-1&\leq \mpl'\Inj(X,r)\leq \mpl'(X,r).
\end{align*}
\end{corollary}
\begin{proof}
If we apply \cref{lem: multipermutation surjection} to the surjective homomorphisms
\[(X,r)\to \Inj(X,r)\to \Ret(X,r),\]
we obtain that $\mpl(\Ret(X,r))\leq \mpl(\Inj(X,r))\leq \mpl(X,r)$ and $\mpl'(\Ret(X,r))\leq \mpl'(\Inj(X,r))\leq \mpl'(X,r)$. It then suffices to note that the inequalities $\mpl(X,r)\leq \mpl(\Ret(X,r))+1$ and $\mpl'(X,r)\leq \mpl'(\Ret(X,r))+1$ hold.
\end{proof}

\begin{proposition}\label{prop: perm skew brace cycle base}
    Let $A$ be a skew brace, $X$ a cycle base of $A$ and let $(X,r)$ denote the restriction of $(A,r_A)$ to $X$. Then $\G(X,r)\cong \Ret(A)$.
\end{proposition}
\begin{proof}
    The universal property of $\G(X,r)$ yields a surjective skew brace homomorphism $f:G(X,r)\to A$, which is injective on the set $\iota(X)$. We claim that $f^{-1}(\Soc(A))=\Gamma$,  with $\Gamma$ the kernel of the canonical skew brace homomorphism $\pi:G(X,r)\to \G(X,r)$, from which the statement then follows. Note that $(X,r)$ is injective, hence $\Gamma=\Soc(G(X,r))$. From this fact, the inclusion from right to left follows. Now let $g\in \G(X,r)\setminus \Gamma$. As $\iota(X)$ is a transitive cycle base of $(G(X,r),+)$ there exist some $x,y\in X$, $x\neq y$, such that $\theta_{(g,0)}(\iota(x))=\iota(y)$ or $\theta_{(0,g)}(\iota(x))=\iota(y)$. As $f(\iota(x))\neq f(\iota(y))$, this implies that $f(g)\notin \Soc(A)$ and this proves the remaining inclusion.
\end{proof}

\begin{corollary}\label{cor: optimized cor mpl}
    Let $(X,r)$ be a solution, then $\G(\Ret(X,r))\cong \Ret(\G(X,r))$.
\end{corollary}
\begin{proof}
    Let $(X,r)$ be a solution, then we know that the canonical image of $X$ in $\G(X,r)$ is a cycle base with its associated solution isomorphic to $\Ret(X,r)$. \cref{prop: perm skew brace cycle base} now yields that $\G(\Ret(X,r))\cong \Ret(\G(X,r))$.
\end{proof}

We now obtain the main result of this section. 
\begin{theorem}\label{theorem: mpl of permutation skew brace is mpl of solution}
Let $(X,r)$ be a solution, then
\begin{equation*}
    \mpl'(X,r)=\mpl(\G{(X,r)}).\label{eq: mpl G and (X,r) 1}
\end{equation*}
\end{theorem}
\begin{proof}
From \cref{cor: optimized cor mpl}, we find that in general $\G(\Ret^n(X,r))\cong\Ret^n(\G(X,r))$. Note that a solution is trivial if and only if its permutation skew brace is the zero brace, from which the equality follows.
\end{proof}
As $\mpl(X,r)$ can vary between $\mpl'(X,r)$ and $\mpl'(X,r)+1$, there is no general way to express $\mpl(\G(X,r))$ directly in terms of $\mpl(X,r)$. This shows the main advantage of $\mpl'(X,r)$ over $\mpl(X,r)$. If $(X,r)$ is indecomposable, we obtain the following result from \cref{prop: mpl indecomposable}.
\begin{corollary}\label{cor: mpl indecomposable solution}
    Let $(X,r)$ be an indecomposable multipermutation solution, then $\mpl(X,r)=\mpl(\G(X,r)).$
\end{corollary}
\begin{corollary}\label{cor: biskew}
    Let $(X,r)$ be an involutive solution. Then, the following conditions are equivalent:
    \begin{enumerate}
        \item $\mpl'(X,r)=2$,
        \item $\G(X,r)$ is a non-trivial $\lambda$-homomorphic bi-skew brace.
    \end{enumerate}
\end{corollary}
\begin{proof}
    This follows from \cref{theorem: mpl of permutation skew brace is mpl of solution} and the fact that for a brace $A$, $\mpl(A)\leq 2$ if and only if $A$ is a $\lambda$-homomorphic bi-skew brace.
\end{proof}
\begin{theorem}\label{theorem: mpl solution (X r) and G(X r)}
Let $(X,r)$ be a multipermutation solution with $|X|>1$, then 
\begin{equation*}
    \mpl(G(X,r))-1\leq \mpl'(X,r)\leq \mpl(G(X,r))
\end{equation*}
If moreover $(X,r)$ is injective, then $\mpl'(X,r)+1=\mpl(G{(X,r)})$.
\end{theorem}
\begin{proof}
If $(X,r)$ is injective, then $\G(X,r)=\Ret(G(X,r))$, hence the second part of the statement follows from \cref{theorem: mpl of permutation skew brace is mpl of solution}. The first part now follows from \cref{cor: multipermutation level of injectivization} and the fact that $G(X,r)=G(\Inj(X,r))$.
\end{proof}
As a consequence, we find the following extension of \cite[Corollary 4.16]{CJKVAV23} and \cite[Theorem 5.15]{gateva2018set}.
\begin{corollary}\label{cor: generalisation GI}
    Let $(X,r)$ be an injective multipermutation solution, then 
    \[\mpl(G(X,r))-1\leq \mpl(X,r)\leq \mpl(G(X,r)).\]
    If moreover, $(X,r)$ satisfies condition $(*)$ then $\mpl(X,r)=\mpl(G(X,r))$.
\end{corollary}
\begin{proof}
    This follows directly from \cref{lem: trivial inequality mpl} and \cref{theorem: mpl solution (X r) and G(X r)}. The second part follows from \cref{prop: mpl square-free} and \cref{theorem: mpl solution (X r) and G(X r)}.
\end{proof}

\section{Generators of a skew brace}\label{section: generators of skew brace}
Let a skew brace $A$ be given. In light of the connection between indecomposable solutions and transitive cycle sets, it is natural to ask whether we can easily characterise when $A$ has a transitive cycle base. More generally, one might search for the minimal number of orbits that a cycle base of $A$ can contain. We start this section by translating this question purely in terms of strong left ideals of $A$. Afterwards, we relate this to the number of generators of the skew brace $A$ and its weight. 
\begin{definition}
    Let $A$ be a skew brace. For a set $X\subseteq A$, the subskew brace, respectively strong left ideal or ideal, of $A$ generated by $X$ is the smallest subskew brace, respectively strong left ideal or ideal, of $A$ containing $X$. If $A$ has a singleton generating set as a skew brace, respectively strong left ideal or ideal, then we say that $A$ is \textit{one-generated} as a skew brace, respectively strong left ideal or ideal. If $A$ is one-generated as a skew brace, we also say that $A$ is a \textit{one-generator skew brace}.
\end{definition}
For some subset $X$ of a skew brace $A$, we denote the subgroup of $(A,+)$ generated by $X$ by $\langle X\rangle_+$.
\begin{lemma}\label{lem: strong left ideal generated by set}
Let $A$ be a skew brace and $X\subseteq A$. The strong left ideal generated by $X$ is $L(X)=\langle \theta_{(a,b)}(x)\mid x\in X, a,b\in A\rangle_+$.
\end{lemma}
\begin{proof}
    It is clear that any $L(X)$ is contained in any strong left ideal containing $X$. Moreover, $\theta_{(a,b)}(L(X))\subseteq L(X)$ for all $a,b\in A$, hence $L(X)$ is a strong left ideal itself.
\end{proof}
From \cref{lem: strong left ideal generated by set} we immediately obtain the following result, which lets us translate the question ``What is the minimal number of orbits contained in a cycle base of $A$?" to ``What is the minimal number of elements that generate $A$ as a strong left ideal?".
\begin{proposition}
    Let $A$ be a non-zero skew brace. If $X$ is a cycle base of $A$ and $Y$ a set of representatives of the orbits in $X$, then $Y$ generates $A$ as a strong left ideal. Conversely, if $Y\subseteq A$ generates $A$ as a strong left ideal, then the union of the orbits of elements in $Y$ form a cycle base of $A$.
\end{proposition}
In particular we find the following corollary.
\begin{corollary}\label{cor: element in transitive cycle base generates skew brace}
    Let $A$ be a skew brace. If $X$ is a transitive cycle base of $A$, then every element of $X$ generates $A$ as a strong left ideal. 
\end{corollary}
In the rest of this section, we study the relation between the generators of $A$ as a skew brace, strong left ideal and ideal. Inspired by results in \cite{Smo18, rump2020one} we first consider multipermutation skew braces. Later, also left nilpotency or annihilator nilpotency appear as a natural assumption. The following theorem generalises \cite[Theorem 5.4]{Smo18} and \cite[Proposition 10]{rump2020one}.
\begin{theorem}\label{theorem: finite mpl strong left implies skew}
Let $A$ be a multipermutation skew brace. If $X$ generates $A$ as a strong left ideal, then $X$ generates $A$ as a skew brace.
\end{theorem}
\begin{proof}
We will prove this claim by induction on the multipermutation level of $A$. If $A$ is a trivial brace, then the claim clearly holds. Now assume that $A$ is not a trivial brace and let $A(X)$ denote the subskew brace of $A$ generated by $X$, the induction hypothesis implies that $A=A(X)+\Soc(A)$. As $X$ generates $A$ as a strong left ideal, we know that $A=L(X)$ with $L(X)$ as in \cref{lem: strong left ideal generated by set}. For any $a,b\in A$, we can write $a=a_1+a_2$ and $b=b_1+b_2$ with $a_1,b_1\in A(X)$ and $a_2,b_2\in \Soc(A)$. This then implies that $\theta_{(a,b)}(x)=\theta_{(a_1,b_1)}(x)\in A(X)$ for all $x\in X$ and therefore $A=L(X)\subseteq A(X)$.
\end{proof}
\begin{corollary}\label{cor: transitive cycle base finite mpl is one generator}
    Let $A$ be a multipermutation skew brace. The following are equivalent:
    \begin{enumerate}
    \item $A$ is one-generated as a skew brace,
    \item $A$ is one-generated as a strong left ideal.
    \end{enumerate}
    In particular, if $x\in A$ generates $A$ as a strong left ideal then it generates $A$ as a skew brace. 
\end{corollary}
We obtain the following corollary which generalises \cite[Theorem 3]{rump2020one}. For simplicity, we say that a solution $(X,r)$ is the cycle base of a skew brace $A$ if $X$ is a cycle base of $A$ and the restriction of $(A,r_A)$ to $X$ is precisely $(X,r)$.
\begin{corollary}
    Let $(X,r)$ be an injective multipermutation solution. The following statements are equivalent.
    \begin{enumerate}
        \item $(X,r)$ is indecomposable,
        \item Every skew brace with $(X,r)$ as a cycle base is generated by any $x\in X$,
        \item Every $x\in X$ generates a skew brace $A$ such that $X$ is the orbit of $x$ in $A$.
        \item There exists a skew brace generated by some $x\in A$, such that $X$ is the orbit of $x$ in $A$.
    \end{enumerate}
\end{corollary}
\begin{proof}
    We first prove the implication from 1 to 2. Recall that every injective solution $(X,r)$ is a cycle base of $G(X,r)$. If moreover $(X,r)$ is multipermutation, then so is $G(X,r)$ by \cref{theorem: mpl solution (X r) and G(X r)}. If $(X,r)$ is indecomposable, then $(X,r)$ is a transitive cycle base of $G(X,r)$ hence $G(X,r)$ is generated by any element in $X$ by \cref{cor: element in transitive cycle base generates skew brace}. By the universal property of $G(X,r)$ it follows that every skew brace $A$ with $(X,r)$ as a cycle base is a homomorphic image of $G(X,r)$ (where the cycle base $X\subseteq G(X,r)$ is mapped to the cycle base $X\subseteq A$). Therefore also $A$ is generated by any $x\in X$.

    The implication from 2 to 3 once again follows from the embedding of $(X,r)$ into $G(X,r)$. The implication from 3 to 4 is trivial. To see that 4 implies 1, note that $X$ is a transitive cycle base of $A$ so $(X,r)$ must be indecomposable.
\end{proof}
Next, we study the relation between sets that generate a skew brace as an ideal and as a strong left ideal. The following result, which was proved in \cite{JKVAV2021}, gives a nice way to determine the minimal number of generators as an ideal, for a large class of skew braces. For a skew brace $A$, we let $\omega(A)$ be the minimal (possibly infinite) cardinality of a subset of $A$ which generates it as an ideal. Also, following \cite{BFP2022,LV2022} for any skew brace $A$ we define $A'$ as the ideal generated by $A^2$ and the commutator subgroup of $(A,+)$, and call this the \textit{commutator} of $A$.
\begin{theorem}\label{theorem: weight}
    Let $A$ be a skew brace with $\omega(A)<\infty$ and satisfying the DCC on ideals. Then $\omega(A)=\omega(A/A^2)=\omega(A/A')$.
\end{theorem}
The following example shows that in general it is not true that if a skew brace is one-generated as an ideal, then it is also one-generated as a strong left ideal. If we want such a result to hold, we thus need to impose extra conditions on $A$.
\begin{example}
     Let $A=\Triv(\Z/2)$ and $B=\Triv(\Z/p\times \Z/p)$ for some odd prime $p$ and consider the semidirect product $C=A\ltimes B$, in the sense of \cite[Corollary 3.36]{SV18}, where $A$ acts by inversion. Explicitly,
    \begin{align*}
        (n,m,l)+(n',m',l')&=(n+n',m+m',l+l'),\\
        (n,m,l)\circ (n',m',l')&=(n+n',m+(-1)^{n}m',l+(-1)^nl'),\\
        \lambda_{(n,m,l)}(n',m',l')&=(n',(-1)^nm',(-1)^nl').
    \end{align*}
    Then we find that $C^2=B$ and thus $(C/C^2,+)\cong \Z/2$ is cyclic, from which it follows that $\omega(A)=1$ by \cref{theorem: weight}. We claim that $C$ is not one-generated as a strong left ideal. If it were, then the image of this generator should generate $C/C^2$ hence it is of the form $(1,l,m)$, with $l,m\in \Z/p$. We can easily see however that for any choice of $l,m\in \Z/p$, the set $\{(n,rm,rl)\mid n\in \Z/2,r\in \Z/p\}$ is a strong left ideal which contains $(1,l,m)$ and which has index $p$. Thus $C$ is not one-generated as a strong left ideal.
\end{example}
\begin{theorem}\label{theorem: left nilpotent generators ideal}
    Let $A$ be a left nilpotent skew brace and $X$ a subset of $A$. If the image of $X$ in $A/A^2$ generates $A/A^2$ as a strong left ideal, then $X$ generates $A$ as a strong left ideal.
\end{theorem}
\begin{proof}
   Let $L=L(X)$ denote the strong left ideal of $A$ generated by $X$. By induction on $n$ we will prove that $L+A^n=A$, which then implies the statement. As the natural image of $X$ in $A/A^2$ generates $A/A^2$ as a strong left ideal, we find that $L+A^2=A$. Now let $n\geq 2$ and assume that $L+A^n=A$. As for $a,b,c\in A$ we have $a*(b+c)=a*b+b+a*c-b$, we find
   \begin{equation*}
       A^2=A*(L+A^n)=A*(A^n+L)\subseteq A^{n+1}+L.
   \end{equation*}
   We know however that $L+A^2=A$, hence $A\subseteq L+A^{n+1}$.
\end{proof}
\begin{corollary}\label{cor: left nilpotent number of generators strong left ideal is number of generators trivialisation}
    Let $A$ be a left nilpotent skew brace of finite weight satisfying the descending chain condition (henceforth DCC) on ideals. The minimal number of generators of $A$ as a strong left ideal coincides with $\omega(A)$.
\end{corollary}
\begin{proof}
     It suffices to prove that $A$ is generated as a strong left ideal by $\omega(A)$ elements. From \cref{theorem: weight} we know that $\omega(A)=\omega(A/A^2)$. As every strong left ideal in $A/A^2$ is an ideal, $\omega(A/A^2)$ is also the minimal numbers of generators of $A/A^2$ as a strong left ideal and by \cref{theorem: left nilpotent generators ideal} we thus obtain a generating set of size $\omega(A)$ which generates $A$ as a strong left ideal.
\end{proof}
\begin{corollary}\label{cor: left nilpotent onegenerated}
    Let $A$ be a left nilpotent skew brace such that $A$ satisfies the DCC on ideals, then the following are equivalent: 
    \begin{enumerate}
        \item $A$ is one-generated as a strong left ideal,
        \item $A$ is one-generated as an ideal,
        \item $(A/A',+)$ is cyclic.
    \end{enumerate}
\end{corollary}
\begin{proof}
    The equivalence of 1 and 2 is clear from \cref{cor: left nilpotent number of generators strong left ideal is number of generators trivialisation}. The equivalence of 2 and 3 follows from \cref{theorem: weight}.
\end{proof}

\begin{lemma}\label{lem: A/A' cyclic left nilpotent}
    Let $A$ be a skew brace such that $(A/A',+)$ is cyclic and $(A,+)$ or $(A,\circ)$ is nilpotent, then $A^2=A'$. 
\end{lemma}
\begin{proof}
    Consider the quotient $A/A^2$ and denote $G$ its underlying group. Then $G$ is nilpotent and $G/G'$ is cyclic, but an easy exercise shows that this implies that $G$ is cyclic. In particular, $A/A^2$ is a trivial brace, hence $A'\subseteq A^2$.
\end{proof}
\begin{proposition}\label{prop: annihilator nilpotent }
    Let $A$ be an annihilator nilpotent skew brace. Then the following are equivalent:
    \begin{enumerate}
        \item $A$ is one-generated as a skew brace,
        \item $A$ is one-generated as a strong left ideal,
        \item $A$ is one-generated as an ideal,
        \item $(A/A^2,+)$ is cyclic.
    \end{enumerate}
    In this case, the following are equivalent for an element $x\in X$:
    \begin{enumerate}
        \item $x$ generates $A$ as a skew brace,
        \item $x$ generates $A$ as a strong left ideal,
        \item $x$ generates $A$ as an ideal,
        \item $x+A^2$ generates $(A/A^2,+)$.
    \end{enumerate}
\end{proposition}
\begin{proof}
    The first part is a consequence of \cref{cor: transitive cycle base finite mpl is one generator}, \cref{cor: left nilpotent onegenerated} and \cref{lem: A/A' cyclic left nilpotent}.

    The second part now follows if we also take into account \cref{theorem: left nilpotent generators ideal}.
\end{proof}
The first part of \cref{prop: annihilator nilpotent } generalises \cite[Corollary 1]{rump2022class}.
In the same paper, in Proposition 5, Rump showed that for one-generated braces with an abelian multiplicative group, the transitive cycle bases are precisely the cosets of $A^2$ which generate $A/A^2$. Note that a finite brace with an abelian multiplicative group is a finite two-sided brace, which implies that it is annihilator nilpotent, so one might expect that a similar result might hold for all annihilator nilpotent braces. The following example shows that even skew braces that are very similar to this class, in this case annihilator nilpotent braces and annihilator nilpotent skew braces with abelian permutation group, do not exhibit a similar feature.
\begin{example}\label{ex: counter 1}
    Let $(A,+)=(\Z/p)^n$, for $p$ a prime and $2\leq n<p$ and let $\phi\in \Aut(A,+)$ be the automorphism given by the Jordan block of size $n$ (where we consider $(\Z/p)^n$ as column vectors). Let $\lambda:(A,+)\to \Aut(A,+)$ be given by $(a_1,...,a_n)\mapsto \phi^{a_n}$. In particular, we find that $\ker \lambda=(\Z/p)^{n-1}\times \{0\}=\{\phi(a)-a\mid a\in A\}$. So from \cite[Theorem 6.6]{ST23} we find a bi-skew brace $(A,+,\circ)$ where $a\circ b=a+\lambda_a(b)$. It is easily seen that $(A,\circ)$ is abelian if and only if $n=2$. Note that $A$ is right nilpotent of class 2, so in particular it is a bi-skew brace. It is left nilpotent as it is of prime power size (see \cite[Proposition 4.4]{CSV19}), hence $A$ is an annihilator nilpotent brace. As $(A,\circ)/\ker \lambda$ is cyclic of order $p$, we see that all transitive cycle bases have size $p$. Because $|A^2|=p^{n-1}$, the transitive cycle bases are cosets of $A^2$ if and only if $n=2$, in this case $(A,\circ)$ is abelian.
\end{example}
\begin{example}
     As for any choice of a prime $p$ and $2\leq n<p$, the skew brace $(A,+,\circ)$ from \cref{ex: counter 1} is a bi-skew brace, we can also consider the skew brace $(A,\circ,+)$, which is still annihilator nilpotent. Now its multiplicative group is always abelian and its additive group is abelian if and only if $n=2$. As the $\lambda$-map of an element $a\in A$ in this skew brace is given by $\lambda_a^{-1}$, with $\lambda_a$ the $\lambda$-map of $a$ in $(A,+,\circ)$, and $(A,\circ,+)^2=(A,+,\circ)^2$, we have the same conclusion as in the previous case:  the transitive cycle bases of $(A,\circ,+)$ are cosets of $(A,\circ,+)^2$ if and only if $n=2$, so when $A$ is a brace.
\end{example}

\section{An application to solutions of finite multipermutation level}\label{section: an application to mp solutions}
In \cite[Proposition 5.1]{JPZ2021}, Jedli\v{c}ka, Pilitowska and Zamojska-Dzienio showed that the automorphisms group of a finite indecomposable involutive solution with abelian permutation group and multipermutation level $2$ coincides with the whole endomorphism semigroup. In the following theorem we show the remarkable fact that an indecomposable multipermutation solution has no subsolutions. As a consequence we generalise the above result on the endomorphism semigroup to all finite indecomposable involutive solutions with finite multipermutation level.
\begin{theorem}\label{nosubsol}
    Let $(X,r)$ be an indecomposable solution of finite multipermutation level, then $(X,r)$ contains no non-trivial subsolutions.
\end{theorem}
\begin{proof}
    Let $(Y,s)=(Y,r_{Y\times Y})$ be a subsolution of $(X,r)$. Note that the defining relations of the additive, respectively multiplicative, group of the structure skew brace $G(Y,s)$ are contained in the defining relations of the additive, respectively multiplicative, group of the structure skew brace $G(X,r)$. Therefore, we get a well-defined skew brace homomorphism $\phi:G(Y,s)\to G(X,r)$. As the image of $\phi$ contains elements of the transitive cycle base $\iota(X)$ of $G(X,r)$ and every element of this transitive cycle base generates $G(X,r)$ by \cref{theorem: finite mpl strong left implies skew}, we find that $\phi$ is surjective. Because $\phi$ is a skew brace homomorphism, it induces a homomorphism 
    $$ \phi':(G(Y,s),+)\rtimes (G(Y,s),\circ)\to (G(X,r),+)\rtimes (G(X,r),\circ),\quad (a,b)\mapsto (\phi(a),\phi(b)),$$
    which factorises the action of $(G(Y,s),+)\rtimes (G(Y,s),\circ)$ through the action of $(G(X,r),+)\rtimes (G(X,r),\circ)$ on $X$. In particular we know that $(G(Y,s),+)\rtimes (G(Y,s),\circ)$ acts transitively. Note that the image of an element under this action is completely expressible by permutations $\sigma_y,\sigma_y^{-1},\tau_y,\tau_y^{-1}$ for $y\in Y$. As $r(Y\times Y)\subseteq Y\times Y$, this implies that $Y=X$.
\end{proof}

\begin{corollary}
    Let $(X,r)$ be an indecomposable multipermutation solution, then every endomorphism of $(X,r)$ is surjective. In particular, if $|X|<\infty$ then every endomorphism of $(X,r)$ is an automorphism.
\end{corollary}
\begin{proof}
    It suffices to recall that the image under an endomorphism of $X$ is a subsolution of $X$.
\end{proof}
The following example shows that the hypothesis on the multipermutation level can not be dropped.

\begin{example}
Let $X:=\{1,2,3,4\}$ and $r$ be the involutive solution given by $\sigma_1:=(3\;4)$, $\sigma_2:=(1\;3\;2\;4)$, $\sigma_3:=(1\;4\;2\;3)$, $\sigma_4:=(1\;2)$, for all $x,y\in X$. Then, the set $Y:=\{1\}$ is a subsolution of $X$.
\end{example}

In \cite{Bonatto22}, Bonatto studied a particular class of left quasigroups called \emph{superconnected} left quasigroups. Recall that a left quasigroup is an algebraic structure with a single binary operation such that all the left multiplications are invertible. Following the terminology of \cite{Bonatto22}, a left quasigroup is said to be \emph{connected} if the group $\mathcal{G}(X)$ generated by the left multiplications acts transitively on $X$ and superconnected if every sub-left quasigroup is connected. Indecomposable involutive solutions give rise to connected left-quasigroups by the operation $x\cdotp y:=\sigma_x^{-1}(y)$ for all $x,y\in X$ (see \cite{rump2005decomposition} for more details). As a further corollary of \cref{nosubsol}, we have the following result.

\begin{corollary}\label{corosol}
    Let $(X,r)$ be an indecomposable involutive non-degenerate solution having finite multipermutation level. Then, the left-quasigroup associated to $X$ has no nontrivial sub-left quasigroup. In particular, $X$ is superconnected.
    \end{corollary}

\section{Automorphism group of solutions}\label{section: automorphism group}

In this section, we investigate the automorphisms group of indecomposable involutive solutions with a regular permutation group. In particular, we recover and extend some results contained in \cite[Section $5$]{JPZ2021}, where the automorphisms group of involutive indecomposable solutions with an abelian permutation group and having multipermutation level $2$ is studied in detail. 
\begin{remark}\label{Arx}
We exclusively consider indecomposable involutive solutions with a regular permutation group in this section. By \cref{theorem: construction solution Bachiller}, it is no restriction to assume for the rest of the section that $(X,r)=(A,r_x)$ with $x$ some element contained in a transitive cycle base of a skew brace $A$. 
\end{remark}
We start by giving a description of the automorphisms group of an indecomposable solution with regular permutation group (not necessarily abelian) by means of left braces. Recall that for a skew brace $A$ and $z\in A$, $t_z$ is the right translation by $z$ in the multiplicative group.
\begin{proposition}\label{propo2.9}
Let $A$ be a brace, $x\in A$ contained in a transitive cycle base and $(X,r):=(A,r_x)$. Then, 
 $$\Aut(X,r)=\{t_z\circ \psi\mid\psi\in \Aut(A,+,\circ),z\in A, \psi(x)=\lambda_z(x)  \}.$$
Moreover, if $X$ has finite multipermutation level, we have that $|\Aut(X,r)|\leq |X|$.
\end{proposition}
\begin{proof}
By \cref{theorem: isomorphism solution Bachiller},  if $(A,r_x)$ and $(A,r_y)$ are two isomorphic solutions with $x,y\in X$ and $F$ is an isomorphism from $(A,r_x)$ to $(A,r_y)$, then there exist $\psi\in Aut(A,+,\circ)$ and $z\in A$ such that $\psi(x)=\lambda_z(y)$ and $F=t_z\circ \psi$. Therefore, the first part of the thesis follows setting $x=y$. For the last part, it is sufficient to note that if $\psi,\psi'\in \Aut(A,+,\circ)$ are such that $\psi(x)=\psi'(x)$, since with our hypothesis it follows from \cref{cor: transitive cycle base finite mpl is one generator} that $A$ is generated by $x$ as a skew brace, from which we obtain that $\psi=\psi'$.
\end{proof}

\begin{remark}\label{remm}
    A well-known result due by Rump \cite{rump2007classification} states that if $A$ is a finite brace with abelian multiplicative group, then, as left brace, it is isomorphic to the direct product of the Sylow $p$-subgroups of its additive group. This fact, together with the previous proposition, implies that the calculation of the automorphisms group of a finite indecomposable involutive solution with abelian permutation group can be reduced to the calculation of the ones having prime-power size. 
\end{remark}

Let $A$ be a skew brace with transitive cycle base $Y$. We define $\Aut(A,Y)$ as the group of all skew brace automorphisms $\psi:A\to A$ such that $\psi(Y)\subseteq Y$. As a further corollary, we have the following. 

\begin{corollary}\label{cor: normal subgroup aut(X r) isomorphic to Soc(A)}
Let $A$ be a brace, $Y$ a transitive cycle base of $A$, $x\in Y$ and $(X,r):=(A,r_x)$. Then $\Aut(X,r)$ has a normal subgroup isomorphic to $(\Soc(A),\circ)$ such that its quotient is isomorphic to $\Aut(A,Y)$.
\end{corollary}
\begin{proof}
    Let $\phi: \Aut(X,r)\to \Aut(A,Y)$ be the map sending $t_z\circ \psi\in \Aut(X,r)$ to $\psi$. This is a well-defined group homomorphism with kernel precisely $\{t_z\mid z\in \Soc(A)\}$. Moreover, the definition of $\Aut(A,Y)$ ensures that $\phi$ is surjective.
\end{proof}

\subsection{Indecomposable solutions with abelian permutation group and multipermutation level $2$}

Now, we consider indecomposable solutions with abelian permutation group and multipermutation level $2$. 

\begin{lemma}\label{lem65}
Let $A$ be a brace with a transitive cycle base $X$. Moreover, suppose that $A$ has multipermutation level $2$. Then $\Aut(A,X)=\{\lambda_a\mid a\in A\}$.
\end{lemma}

\begin{proof}
Since by \cref{cor: mpl indecomposable solution} and \cref{cor: biskew} $A$ is a bi-skew brace, we find that $\lambda_a$ is a skew brace automorphism for each $a\in A$. Because in particular $A$ has finite multipermutation level, we know by \cref{cor: transitive cycle base finite mpl is one generator} that every element $x\in X$, generates $A$ as a brace. Hence for every $\psi\in \Aut(A,X)$ we find $\psi(x)=\lambda_a(x)$ for some $a$, hence $\psi=\lambda_a$.
\end{proof}

In this context, the description of the automorphisms group of an indecomposable solution can be simplified, as we can see in the following result.

\begin{corollary}\label{coro}
Let $A$ be a brace, $x\in A$ contained in a transitive cycle base and $(X,r):=(A,r_x)$. Moreover, suppose that $\mpl(X,r)=2$. Then, 
$$\Aut(X,r)=\{t_a\circ \lambda_a  \mid a\in A\}.$$
\end{corollary}

\begin{proof}
The proof follows by \cref{cor: mpl indecomposable solution}, \cref{propo2.9}, and \cref{lem65}.
\end{proof}

In the next theorem, which is the main result of the section, we provide the structure of the automorphisms group of an indecomposable involutive solution with abelian permutation group and multipermutation level $2$ in terms of the permutation brace.

\begin{theorem}\label{isoaut}
Let $A$ be a brace with abelian multiplicative group, $x\in A$ contained in a transitive cycle base and $(X,r):=(A,r_x)$. Moreover, suppose that $\mpl(X,r)=2$. Then, its automorphism group $\Aut(X,r)$ is isomorphic to the additive group $(A,+)$.
\end{theorem}

\begin{proof}
Let $\Phi$ be the function from $\Aut(X,r)$ to $(A,+)$ given by $\Phi(t_a\circ \lambda_a):=a^- $ for all $a\in A$ (recall that, by \cref{coro}, with our hypothesis every element of $\Aut(X,r)$ is of the form $t_a\circ \lambda_a $). Then, $\Phi$ is a bijection and
$$\Phi(t_a\circ \lambda_a\circ t_b\circ \lambda_b)=\Phi(t_{a\circ \lambda_a(b)}\circ \lambda_{a\circ b})=(a\circ \lambda_a(b))^- $$
for all $a,b\in B$. Now, by \cref{cor: biskew} $\lambda_a\in Aut(A,\circ)$ and by hypothesis $(A,\circ)$ is abelian, hence we have that 
$$ (a\circ \lambda_a(b))^-=a^-\circ \lambda_a(b^-)=a^-+b^-=\Phi(t_a\circ \lambda_a)+\Phi(t_b\circ \lambda_b)$$
therefore the statement follows.
\end{proof}

As a first application, we can give a shorter proof of \cite[Proposition 5.1]{JPZ2021}.

\begin{corollary}\label{regab2}
Let $A$ be a brace with abelian multiplicative group, $x\in A$ contained in a transitive cycle base and $(X,r):=(A,r_x)$. Moreover, suppose that $\mpl(X,r)=2$. Then, its automorphisms group $\Aut(X,r)$ is a regular abelian group.
\end{corollary}

\begin{proof}
Since $t_a(\lambda_a(0))=a$ for all $a\in A$, we have that $\Aut(X,r)$ is transitive and, by the previous theorem, $\Aut(X,r)$ is an abelian group with $|\Aut(X,r)|=|X|$.
\end{proof}

In \cite[Proposition 5.3]{JPZ2021}, necessary and sufficient conditions on the size of an indecomposable involutive solution $(X,r)$ with $\mpl(X,r)=2$ and cyclic permutation group were given to state when $\Aut(X,r)$ is cyclic. In the following, we give the exact structure of $\Aut(X,r)$ when $(X,r)$ has cyclic permutation group.

\begin{corollary}\label{mpl22}
 Let $(X,r)$ be an indecomposable involutive solution with cyclic permutation group having size $p^n$ for some prime number $p$. Moreover, suppose that $\mpl(X,r)=2$. Then, $\Aut(X,r)$ is isomorphic to the cyclic group of size $p^n$ if $(p,n)\neq (2,2)$ and is isomorphic to the Klein group otherwise.
\end{corollary}

\begin{proof}
Since by \cref{isoaut} $\Aut(X,r)$ is isomorphic to $(\G(X,r),+)$, by \cite[Proposition 5.4]{bachiller2016solutions} it is cyclic if $(p,n)\neq (2,2)$. If $p=n=2$, the result follows by the classification of braces having size $p^2$ given in \cite{bachiller2015classification}.
\end{proof}

\begin{corollary}
 Let $(X,r)$ be an indecomposable involutive solution with abelian non-cyclic permutation group having size $p^n$ for some prime number $p$. Moreover, suppose that $\mpl(X,r)=2$. Then, $\Aut(X,r)$ is cyclic if and only if $(p,n)=(2,2)$. 
\end{corollary}

\begin{proof}
If $\Aut(X,r)$ is cyclic, by \cref{isoaut} the additive group of $\G(X,r)$ is cyclic and since $\G(X,r)$ is a bi-skew brace, we have that $(\G(X,r),\circ,+) $ is a brace, hence by \cite[Proposition 5.4]{bachiller2016solutions} we must have $p=n=2$. The converse follows by the classification of left braces having size $p^2$ given in \cite{bachiller2015classification}.
\end{proof}

 We close the section showing a limitation for the structure of $\Aut(X,r)$.

\begin{corollary}
 Let $(X,r)$ be an indecomposable involutive solution with abelian permutation group. Moreover, suppose that $\mpl(X,r)=2$. Then, $\Aut(X,r)$ can be generated by at most two elements.
\end{corollary}

\begin{proof}
Let $(\G(X,r),+,\circ)$ be the permutation brace of $X$. Then, by \cref{cor: biskew} $(\G(X,r),+,\circ)$ is a bi-skew brace of multipermutation level $2$, moreover a transitive cycle base of $(\G(X,r),+,\circ)$ also is a transitive cycle base of $(\G(X,r),\circ,+)$, and this implies that $(\G(X,r),\circ,+)$ is the permutation brace of an indecomposable solutions having multipermutation level $2$. Therefore, by \cite[Proposition 4.1]{JPZ2021}, $(\G(X,r),+)$ can be generated by at most two elements, hence the thesis follows by \cref{isoaut}.
\end{proof}

\subsection{Indecomposable involutive solutions with cyclic permutation group}

In this section we consider the automorphisms group of indecomposable involutive solutions with cyclic permutation group. If the permutation brace of an indecomposable involutive solution $(X,r)$ is \emph{cyclic}, i.e. has cyclic additive group, then we can give a more explicit description of $\Aut(X,r)$.

\begin{proposition}
Let $A$ be a cyclic brace, $x\in A$ contained in a transitive cycle base and $(X,r):=(A,r_x)$. Then, its automorphisms group $\Aut(X,r)$ is equal to the set 
 $$\Aut(X,r)=\{t_a\circ \lambda_a \mid a\in A,\lambda_a\in Aut(A,\circ)  \}.$$
\end{proposition}

\begin{proof}
Under our hypothesis, we have that $x$ is a generator of $(A,+)$: indeed, if the additive subgroup generated by $x$ is equal to $(mA,+)$, for some $m\in \mathbb{N}$, and $Z$ is the transitive cycle base containing $x$, we obtain that $A=\langle Z\rangle_+\subseteq mA$, hence necessarily $m=1$.
Therefore, if $\psi\in Aut(A,+,\circ)$ and $a\in A$ are such that $\lambda_a(x)=\psi(x)$, then $\psi=\lambda_a$ and hence the statement follows by \cref{propo2.9}.
\end{proof}
Before giving the following results we recall that, if $(A,+,\circ)$ is a brace having prime-power size $p^n$, with cyclic additive group (which we indentify with $\mathbb{Z}/p^n$) and abelian multiplicative group, then by \cite{rump2007braces} there exists $k\in \mathbb{N}$ such that $a\circ b=a+b+p^kab$ for all $a,b\in A$.

In \cite[Section $5$]{JPZ2021} it was showed that an indecomposable involutive solution with $\mpl(X,r)=2$ and with abelian permutation group has abelian automorphisms group and the ones having both a cyclic automorphisms group and cyclic permutation group were determined. In the following, we extend the calculation of $\Aut(X,r)$ for indecomposable involutive solutions with cyclic permutation group and arbitrary multipermutation level.   

\begin{theorem}
Let $A$ be a brace having size $p^n$ in which both the additive and the multiplicative group are cyclic, $x\in A$ is contained in a transitive cycle base and $(X,r):=(A,r_x)$. Moreover, let $k\in \mathbb{N}$ be such that $a\circ b:=a+b+abp^k$ for all $a,b\in A $. Then, $\Aut(X,r)$ is isomorphic to $(A,+)$ if and only if $k=0$ or $k\geq \frac{n}{2} $ and it is isomorphic to $(p^{n-2k}A,+)$ otherwise. 
\end{theorem}

\begin{proof}
    Let $\Phi$ be the function from $\Aut(X,r)$ to $(A,+)$ given by $\Phi(t_a\circ \lambda_a):=a^- $ for all $a\in A$. As in \cref{isoaut}, one can show that $\Phi$ is a homomorphism. Moreover, $\Phi$ is injective. Now, by a standard calculation we have that if $z\in A$, then $\lambda_z\in Aut(A,\circ)$ if and only if $k=0$, $k\geq \frac{n}{2}$ or $p^{n-2k}$ divides $z$. 
    From this fact, it follows that the image of $\Phi$ is $A$ if and only if $k=0$ or $k\geq \frac{n}{2}$, otherwise $\Phi(A)$ is equal to $(p^{n-2k}A,+)$.
\end{proof}

Since by \cite[Proposition 5.4]{bachiller2016solutions}, except the case $p^n=4$, every brace of prime power size $p^n$ and with cyclic multiplicative group has cyclic additive group, the previous theorem, together with \cref{mpl22} (to cover the case $p^n=4$), allows to compute the isomorphism class of the automorphisms group of every indecomposable involutive solution with cyclic permutation group.

\begin{corollary}\label{abelcyc}
   Let $(X,r)$ be an indecomposable involutive solution with cyclic permutation group and size $p^n$, for a prime number $p$. Then, $\Aut(X,r)$ is cyclic of order $|X|$ if $(|X|,\mpl(X,r))\neq (4,2)$ and is isomorphic to the Klein group otherwise.
\end{corollary}

\begin{proof}
If $(|X|,\mpl(X,r))\neq (4,2)$, by \cite[Proposition 5.4]{bachiller2016solutions} $(\G(X),+,\circ)$ is a brace in which both the additive and the multiplicative group are cyclic, hence the thesis follows by the previous theorem. If $(|X|,\mpl(X,r))= (4,2)$, then the thesis follows by \cref{mpl22}.
\end{proof}

 In \cite[Proposition 5.1]{JPZ2021} it was shown that, if $X$ is an indecomposable solution with an abelian permutation group and $\mpl(X,r)\leq 2$, then $\Aut(X,r)$ is an abelian group that acts transitively on $X$. We conclude the section stating that if in addition $X$ has cyclic permutation group, even with no restriction on the multipermutation level the abelianity of $\Aut(X,r)$ can be showed. 

\begin{corollary}
   Let $(X,r)$ be an indecomposable involutive solutions with cyclic permutation group. Then, $\Aut(X,r)$ is an abelian group; moreover, it acts transitively on $X$ if and only if $X$ has multipermutation level at most $2$.
\end{corollary}

\begin{proof}
The first part of the statement follows by \cref{remm} and \cref{abelcyc}; the second part follows by \cref{propo2.9} and \cref{regab2}.
\end{proof}

\section{Finite indecomposable involutive solutions with abelian permutation group}\label{section: finite indecomposable solutions}
In this section we study, and up to some extent classify and enumerate, indecomposable solutions with abelian permutation group. The starting point for this is \cref{prop: bijective correspondence pairs and solutions}, which is a slight variation of \cite[Corollary 2]{rump2022class}.
    Let $\AbBr$ be the category with as objects braces $A$ with an abelian multiplicative group, with a distinguished transitive cycle base $X$. We denote such an object by $(A,X)$. Morphisms $(A,X)\to (B,Y)$ are brace morphisms $f:A\to B$ such that $f(X)\subseteq Y$.
    Also, let $\AbSol$ be the category formed by indecomposable involutive solutions with abelian permutation group and homomorphisms of solutions. Note that for an object $(A,X)\in \AbBr$, $\Aut(A,X)$ coincides with the definition in the previous section.
\begin{proposition}\label{prop: bijective correspondence pairs and solutions}
    There exists a bijective correspondence between isomorphism classes of $\AbSol$ and isomorphism classes in $\AbBr$. Moreover, the size and multipermutation level of objects are preserved under this correspondence.

    Under this correspondence, an isomorphism class in $\AbBr$ with as representative $(A,X)$ is mapped to the isomorphism class of the solution $(A,r_x)$ for an arbitrary choice of $x\in X$. 
\end{proposition}
\begin{proof}
    Recall from \cref{theorem: construction solution Bachiller} that every indecomposable involutive solution with abelian permutation group can be obtained as $(A,r_x)$ with $(A,X)\in \AbBr$ and $x\in X$. Moreover, it follows from \cref{theorem: isomorphism solution Bachiller} that the isomorphism class of $(A,r_x)$ does not depend on the choice of $x\in X$. Also by \cref{theorem: isomorphism solution Bachiller} we find that the solutions associated to $(A,X),(B,Y)\in \AbBr$ are isomorphic if and only if $(A,X)$ and $(B,Y)$ are isomorphic. This correspondence clearly preserves the size of objects and \cref{theorem: mpl of permutation skew brace is mpl of solution} and \cref{cor: mpl indecomposable solution} ensure that also the multipermutation level is preserved.
\end{proof}
Because of the following lemma we can restrict our study of finite objects $(A,X)\in \AbBr$ to those of prime power order.
\begin{lemma}\label{lem: product of prime power braces}
    Every object $(A,X)\in \AbBr$ of finite size $m$ is isomorphic to $\prod_{p|m} (A_p,X_p)$ where $A_p$ is the Sylow $p$-subgroup of $(A,+)$ and $X_p$ is the projection of $X$ onto $A_p$.
\end{lemma}
\begin{proof}
    This follows directly from \cite[Proposition 3]{rump2007classification} and \cite[Proposition 3]{rump2022class}.
\end{proof}

Recall from \cref{theorem: finite mpl strong left implies skew} that multipermutation braces with an abelian multiplicative group that admit a transitive cycle base are necessarily one-generated as a brace. Under the correspondence between Jacobson radical rings and two-sided braces, being one-generated as a ring is generally not the same as being one-generated as a skew brace; the latter is a weaker notion as a subring of a Jacobson radical ring is a monoid but not necessarily a group for the operation $\circ$. But for a nil, so in particular a nilpotent, two-sided brace $A$ these two notions coincide as $a^-=\sum_{i=1}^\infty (-a)^i$ for any $a\in A$. 

Let $A$ be a one-generated multipermutation brace with abelian multiplicative group with generator $x\in A$. As $A$ is one-generated as a ring, every element $a\in A$ is of the form $\sum_{i=1}^na_ix^i$, for some $n\geq 0$ and $a_i\in \Z$, where $x^n$ is the $*$-product of $n$ occurrences of $x$. More generally every element $a\in A^k$ can be written as $\sum_{i=k}^na_ix^n$ for some $n\geq 0$ and $a_i\in \Z$, or equivalently $A^k=\{x^{k-1}*a\mid a\in A\}$ for $k>1$. As $*$-multiplication by $x$ is an endomorphism of $(A,+)$, we obtain the following result.
\begin{lemma}\label{lemma: chain}
    Let $A$ be a brace generated by $x\in A$ with abelian multiplicative group. Then we have a chain of surjective group homomorphisms
    \[(A/A^2,+)\to (A^2/A^3,+)\to (A^3/A^4,+)\to ...\]
    where for all $i\geq 1$, the map $A^k/A^{k+1}\to A^{k+1}/A^{k+2}$ is given by $a+A^{k+1}\mapsto x*a +A^{k+2}$.
\end{lemma}
\begin{definition}
For a finite brace $A$ with abelian multiplicative group, we say that $A$ is of type $(m_1,...,m_n)$, for $m_1,...,m_n\geq 1$, if $|A^i/A^{i+1}|=m_i$ for $1\leq i<n$ and $A^n+1=0$. We define $\AbBr(m_1,...,m_n)$ as the full subcategory of $\AbBr$ consisting of objects $(A,X)$ with $A$ of type $(m_1,...,m_n)$.
\end{definition}
From \cref{lemma: chain} it follows directly that if a multipermutation brace $A$ is of type $(m_1,...,m_n)$, then necessarily $m_{i+1}|m_i$ for $0\leq i<n$. Also, it is well-known that finite Jacobson radical rings are nilpotent, so every finite object in $\AbBr$ is of some type. In particular, if $|A|=p^d$ for some prime $p$, then $A$ is of type $(p^{d_1},...,p^{d_n})$ for some $d_1\geq d_2\geq ...\geq d_n\geq 0$ such that $d=d_1+...+d_n$. 
\begin{definition}
    For a finite multipermutation solution $(X,r)$, we say that $(X,r)$ is of type $(m_1,...,m_n)$, for $m_1,...,m_n\geq 1$, if $|\Ret^{i-1}(X,r)|/|\Ret^{i}(X,r)|=m_i$ for $1\leq i< n$ and $|\Ret^n(X,r)|=1$. We define $\AbSol(m_1,...,m_n)$ as the full subcategory consisting of objects $(X,r)\in \AbSol$ such that $(X,r)$ is of type $(m_1,...,m_n)$.
\end{definition}
The following result generalises the observation of Rump that $|A^2||\Soc(A)|=|A|$ for a finite brace $A$ with a cyclic multiplicative group; this was remarked in the introduction of \cite{rump2022class} and follows from equation $(15)$ and Proposition 9 of \cite{rump2019classification}.
\begin{lemma}\label{lem: relation socle and powers}
    Let $A$ be a finite one-generated brace with abelian multiplicative group, then $|Soc_k(A)||A^{k+1}|=|A|$ or equivalently $|A^k|=|\Ret^{k-1}(A)|$ for all $k\geq 0$.
\end{lemma}
\begin{proof}
We prove by induction that for all $k\geq 1$ and $a\in A$, $x^k*a=0$ if and only if $a\in \Soc_k(A)$. For $k=1$, note that $x^k*a=0$ if and only if $a*A=A*a=0$, the latter is equivalent to $a\in \Soc(A)$. Now assume that the statement holds for $k\geq 1$. Then $x^{k+1}*a=x^k*(x*a)=0$ if and only if $x*a\in \Soc_k(A)$. However, as $x+\Soc_k(A)$ generates $A/\Soc_k(A)$, the case $k=1$ yields that the latter is equivalent to $a+\Soc_k(A)\in \Soc(A/\Soc_k(A))$ hence $a\in \Soc_{k+1}(A)$. Now notice that $x^k*a=0$ if and only if $a$ is contained in the kernel of the surjective homomorphism $(A,+)\to (A^{k+1},+)$ given by multiplication by $x^k$. As $A^{k+1}=\{x^{k}*a\mid a\in A\}$, the statement now follows.
\end{proof}
\begin{proposition}
    Let $m_1,...,m_n\geq 1$. The bijective correspondence from \cref{prop: bijective correspondence pairs and solutions} restricts to a bijective correspondence between isomorphism classes of $\AbSol(m_1,...,m_n)$ and isomorphism classes of $\AbBr(m_1,...,m_n)$.
\end{proposition}
\begin{proof}
    This follows directly from \cref{cor: optimized cor mpl} and \cref{lem: relation socle and powers}.
\end{proof}
\begin{corollary}
    Let $m_1,...,m_n\geq 1$ and $(X,r)$ in $\AbSol(m_1,...,m_n)$, then $$|\Aut(X,r)|=m_1|\Aut(\G(X,r),X')|$$ with $X'$ the image of $X$ in $\G(X,r)$.
\end{corollary}
\begin{proof}
    From \cref{cor: normal subgroup aut(X r) isomorphic to Soc(A)} we know that $\Aut(X,r)$ has a normal subgroup isomorphic to $\Soc(\G(X,r))$ such that the quotient is isomorphic to $\Aut(\G(X,r),X')$. As $m_1=|A/A^2|$, which is in turn equal to $|\Soc(A)|$ by \cref{lem: relation socle and powers}, the statement follows.
\end{proof}
Note that if we want to study finite objects of $\AbBr$, then as a direct consequence of \cref{lem: product of prime power braces} it suffices to study $\AbBr(p^{d_1},...,p^{d_n})$ for all primes $p$ and $d_1\geq ...\geq d_n\geq 0$. This is precisely the goal of the rest of this section. In the remainder of this section, we fix the notation that $n$ is a non-zero positive integer, $p$ a prime, $d_1,...,d_n$ are integers such that $d_1\geq ...\geq d_n\geq 0$ and $d=\sum_{i=1}^nd_i$.

\begin{example}
Consider the ring $\Z[x]/(x^{n+1})$. Let $F_n$ denote its subring generated by $x$. Then clearly $F_n$ is nilpotent, in particular $F_n^{n}\neq 0$ but $F_n^{n+1}=0$. It follows that $F_n$ is a one-generated brace with abelian multiplicative group. The orbit of $x$ is easily seen to be the coset $x+F_n^2$ and is a transitive cycle base of $F_n$. 
\end{example}

From now on, the object $(F_n,x+F_n^2)\in \AbBr$ will be denoted by $F_n^*$. If $I$ is an ideal of $F_n$, then the image of $x+F_n^2$ in $F_n/I$ is a transitive cycle base of $F_n/I$ and $F_n/I$ together with this cycle base will be denoted by $(F_n/I)^*$.

The following proposition is easily verified and shows that $F_n$ can be seen as the free one-generator brace of nilpotency class $n$ and with abelian multiplicative group.
\begin{proposition}\label{prop: free one-generator ring}
    Let $A$ be a one-generator two-sided brace with multipermutation level at most $n$ and $y\in A$ a generator, there is a unique surjective brace homomorphism $f:F_n\to A$ mapping $x$ to $y$.
\end{proposition}

\begin{corollary}\label{cor: pair isomorphic to quotient of F_n}
    Let $(A,Y)\in \AbBr$ with $A$ of multipermutation level at most $n$, then there exists a surjective homomorphism $f:F_n^*\to (A,Y)$. In particular, $(A,Y)$ is isomorphic to $(F_n/\ker f)^*$.
\end{corollary}

\begin{corollary}\label{corn}
The permutation group associated to an indecomposable involutive solution of multipermutation level $n$ is generated by at most $n$ elements.
\end{corollary}
\begin{proof}
    As $(F_n^k/F_n^{k+1},\circ)\cong\Z$ for all $1\leq k\leq n$ and $(F_n,\circ)$ is abelian, we find that $(F_n,\circ)$ is free abelian of rank $n$. The result then follows by \cref{cor: pair isomorphic to quotient of F_n}.
\end{proof}
\begin{remark}
For $n=2$ the previous corollary was proved in \cite{JPZ2021}.
\end{remark}

From \cref{cor: pair isomorphic to quotient of F_n} it follows that every finite object in $\AbBr(p^{d_1},...,p^{d_n})$ is isomorphic to $(F_n/I)^*$ for some ideal $I$. So if we want to study finite objects in $\AbBr(p^{d_1},...,p^{d_n})$, it suffices study those of the form $(F_n/I)^*$. This now yields two natural questions: can we determine all ideals $I$ of $F_n$ such that $(F_n/I)^*\in \AbBr(p^{d_1},...,p^{d_n})$ and can we determine when two such ideals give an isomorphic quotient? We start by providing an answer to the first question. 

We define group endomorphisms $s^+,s^-:\Z^n\to \Z^n$ by $s^+(v_1,...,v_n)=(0,v_1,...,v_{n-1})$ and $s^{-}(v_1,...,v_n)=(v_2,...,v_n,0)$.

\begin{definition}\label{def: M}
    Let $\M{(p^{d_1},...,p^{d_n})}$ be the set of all $n\times n$-matrices $M$ such that
\begin{enumerate}
    \item $M$ is upper triangular,
    \item $M$ contains only positive integer elements,
    \item Every diagonal element of $M$ is strictly greater than every other element in the same column,
    \item For every $k$ such that $1\leq k\leq n$, the image of the $k$-th row of $M$ under $s^+$ is contained in the subgroup of $\Z^n$ generated by $k+1$-th until $n$-th row,
    \item The diagonal of $M$ is $(p^{d_1},...,p^{d_n})$. 
\end{enumerate}
\end{definition}

\begin{proposition}\label{prop: matrix ideal}
    There exists a bijective correspondence between matrices in $\M{(p^{d_1},...,p^{d_n})}$ and ideals of $I$ in $F_n$ such that $F_n/I$ is of type $(p^{d_1},...,p^{d_n})$.
\end{proposition}
\begin{proof}
    Let $\M'$ be the class of $n\times n$-matrices satisfying $(1)$-$(3)$ of \cref{def: M} and such that their diagonal contains only non-zero elements. It is easily seen that every integer $n\times n$-matrix of rank $n$ is row equivalent to a unique matrix in $\M'$. It follows from basic techniques from linear algebra that there exists a bijective correspondence between matrices in $\M'$ and subgroups $I$ of $\Z^n$ of finite index, where to a matrix $M=(m_{i,j})\in \M'$ we associate the subgroup of $\Z^n$ generated by the rows of $M$. As the elements $x,x^2,...,x^n$ form a basis of $(F_n,+)$, we obtain a correspondence between matrices in $\M'$ and finite subgroups of $(F_n,+)$.

     We claim that $I_M$ is an ideal if and only if $M$ satisfies $(4)$. Let $f_1,...,f_n\in F_n$ be the generators of $I_M$ associated to the rows of $M$. In order for $I_M$ to be an ideal, we need that $x*f_k\in I_M$ for all $1\leq k\leq n$. As the first $k$ coordinates of $f_k$ are 0, we find that the latter is equivalent to $x*f_k\in \langle f_k+1,...,f_n\rangle_+$. Because $*$-multiplying $f_k$ by $x$ is the same as shifting its coordinates to the right, we find the wanted equivalence. 

    At last we prove that the lower series of $F_n/I$ is of type $(p^{d_1},...,p^{d_n})$ if and only if $M$ satisfies condition $(5)$. For this it suffices to note that $F_n^i$ are all elements which are 0 on the first $i-1$ coordinates, hence the additive group of $(F_n/I)^i/(F_n/I)^{i+1}$ is isomorphic to $\Z/m_{i,i}$, from which the last part of the statement follows.
\end{proof}
\begin{theorem}\label{theorem: every solution in D is of the form...}
    Let $M\in \M{(p^{d_1},...,p^{d_n})}$ and let $A=F_n/I_M$. Then $(A,r_{x+I_M})$ is a solution in $\AbSol{(p^{d_1},...,p^{d_n})}$ and every solution in $\AbSol{(p^{d_1},...,p^{d_n})}$ is isomorphic to such a solution.
\end{theorem}
\begin{proof}
    This follows by \cref{cor: pair isomorphic to quotient of F_n} and \cref{prop: matrix ideal}.
\end{proof}
\begin{remark}\label{rem: explicit form of solution}
    If we identify representatives of elements in $A$ with their coordinates with respect to the basis $x,...,x^n$ of $(F_n,+)$, then we find for $a=(a_1,...,a_n),b=(b_1,...,b_n)\in F_n$ that the $\sigma$-map of the solution $(A,r_{x+I_M})$ in \cref{theorem: every solution in D is of the form...} is given by 
    $\sigma_{a+I_M}(b+I_M)=(c_1(a,b),...,c_n(a,b))+I_M$ with $c_1(a,b)=b_1+1$ and 
    \begin{align*}
        c_i(a,b)=a_{i-1}+b_{i-1}+b_i+\sum_{\substack{1\leq k< i}}a_kb_{i-k}.
    \end{align*}
\end{remark}

\begin{proposition}\label{prop: size M}
    The set $\M{(p^{d_1},...,p^{d_n})}$ has size $p^{d-d_1}$.
\end{proposition}
\begin{proof}
We will prove this by induction on $n$. If $n=1$ the statement is clear. Now let $n\geq 2$. Let $M'\in \M{(p^{d_2},...,p^{d_n})}$. We will count how many matrices $M\in \M{(p^{d_1},...,p^{d_n})}$ are such that if we remove the first row and column of $M$, we obtain $M'$. 

Let $f_i\in \Z^n$ be the $i$-th row of such $M$. Then the $f_2,...,f_n$ are completely determined by $M'$, so we need to compute the number of choices of $f_1$. Let us consider what the conditions in \cref{def: M} imply on $f_1$: condition $(1)$ does not impose restrictions on $f_1$, condition $(4)$ means that $s^{+}(f_1)\in \langle f_2,...,f_n\rangle_+$ and condition $(5)$ means that the first coordinate of $f_1$ is $p^{d_1}$. Conditions $(2)$ and $(3)$ mean that we can determine $f_1$ up to equivalence in $\Z^n/\langle f_2,...,f_n\rangle_+$, as the conditions then ensure a unique choice of representative. Note that considering $f_1$ up to this equivalence does not conflict in any way with the restrictions imposed by $(4)$ and $(5)$.

Clearly $H=\langle s^{-}(f_3),...,s^{-}(f_n)\rangle+\Z x^n$ is the largest subgroup of $\Z^n$ such that $s^+(H)\subseteq \langle f_3,...,f_n\rangle_+$ so in particular $\langle f_2,...,f_n\rangle_+\subseteq H$. Then the condition $s^+(f_1)\in \langle f_2,...,f_n\rangle_+$ is equivalent to $f_1\in p^{d_1-d_2}s^{-}(f_2)+H$. So the number of choices of $f_1$ is precisely the number of elements in the coset $p^{d_1-d_2}s^-(f_2)+H$ modulo $\langle f_2,...,f_n\rangle_+$, but this is precisely the index $|H:\langle f_2,...,f_n\rangle_+|$, which we can calculate as $$\frac{|s^+(F_n):\langle f_2,...,f_n\rangle_+|}{|s^+(F_n):H| }=\frac{p^{d-d_1}}{p^{d-d_1-d_2}}=p^{d_2}.$$ We conclude that $|\M{(p^{d_1},...,p^{d_n})}|=p^{d_2}|\M{(p^{d_2},...,p^{d_n})}|$ which by the induction hypothesis is $p^{d-d_1}$.
\end{proof}
\begin{corollary}
    $\AbSol(p^{d_1},...,p^{d_n})$ contains at most $p^{d-d_1}$ isomorphism classes.
\end{corollary}
\begin{proof}
    This follows directly from \cref{theorem: every solution in D is of the form...} and \cref{prop: size M}.
\end{proof}
Now that we know how to describe and count the finite ideals of $F_n$, we focus on the second question that arose earlier; we describe the isomorphism classes of quotients of $F_n$ through orbits of a certain group action. We first treat general quotients and then restrict to a more specific case as above.

\begin{lemma}\label{lem: automorphisms of F_n}
    For every $y\in x+F_n^2$, there exists a unique automorphism of $F_n^*$ mapping $x$ to $y$.
\end{lemma}
\begin{proof}
Let $y\in x+F_n^2$. It is clear that $y^n=0$, hence there exists a unique ring endomorphism $\phi_y:F_n\to F_n$ with $\phi_y(x)=y$. By \cref{prop: annihilator nilpotent }, $y$ generates $F_n$ as a brace, hence $\phi_y$ is surjective. Because $(F_n,+)$ is free of finite rank, it follows that $\phi_y$ is an automorphism of $F_n$, which by construction is an automorphism of $F_n^*$ as well.
\end{proof}
\begin{corollary}\label{cor: lifting of isomorphisms}
    Let $I$ and $J$ be ideals of $F_n$ and $\phi:(F_n/I)^*\to (F_n/J)^*$ an isomorphism, then there exists an automorphism $\hat{\phi}:F_n^*\to F_n^*$ such that $\hat{\phi}(I)= J$ and $\hat{\phi}$ is a lifting of $\phi$ in the sense that the following diagram commutes.
    \[
\begin{tikzcd}
F_n \arrow[d] \arrow[r, "\hat{\phi}"] & F_n \arrow[d] \\
F_n/I \arrow[r, "\phi"]               & F_n/J        
\end{tikzcd}\]
\end{corollary}
\begin{proof}
Let $y\in \phi(x+I)$ for some $y\in x+F_n^2$. In particular, this implies that for any element $\sum_{i=1}^na_ix^i\in I$, where $a_i\in \Z$, if and only if $\sum_{i=1}^na_iy^i\in J$. Now define $\hat{\phi}:F_n\to F_n$ as the automorphism mapping $x$ to $y$, which exists by the previous lemma. Then in particular, $\hat{\phi}(I)\subseteq I=J$, and thus $\hat{\phi}$ fits in the above diagram.
\end{proof}

Consider the action of $\Aut(F_n^*)$ on ideals of $F_n$ where $\phi\in \Aut(F_n^*)$ maps an ideal $I$ to $\phi(I)$. We then obtain the following result.
\begin{proposition}\label{prop: orbits and isomorphism classes general}
    There is a bijective correspondence between isomorphism classes of quotients of $F_n^*$ and orbits of ideals of $F_n$ under the action by $\Aut(F_n^*)$. Under this correspondence, the orbit of an ideal $I$ is mapped to the isomorphism class of $(F_n/I)^*$.
\end{proposition}
\begin{proof}
    Let $I,J$ be ideals of $F_n$. Assume that there exists an isomorphism $\theta:(F_n/I)^*\to (F_n/J)^*$. Then using \cref{cor: lifting of isomorphisms} we find that $\phi$ lifts to an automorphism $\hat{\phi}$ of $F_{n}^*$ such that $\hat{\phi}(I)=J$. Conversely, any automorphism of $F_n^*$ mapping $I$ to $J$ induces an isomorphism between $F_n/I$ and $F_n/J$. 
\end{proof}

For counting purposes, it is more convenient to consider a slight variation of this action such that we have a finite group acting on a finite set. We do so by restricting to ideals $I$ of $F_n$ such that $F_n/I$ is of type $(p^{d_1},...,p^{d_n})$. Let $(p^dx)$ be the ideal of $F_n$ generated by $p^d x$ and let $F_{n,p^d}=F_n/(p^dx)$. Ideals of $F_n$ of index $p^d$ always contain $p^dx$ so they are in correspondence with ideals of $F_{n,p^d}$ of index $p^d$. By abuse of notation we will also denote the image of $x$ in $F_{n,p^d}$ by $x$ and $F_{n,p^d}^*$ is short for the object in $\AbBr$ consisting of $F_{n,p^d}$ and the transitive cycle base containing $x$. It is clear that every $(A,X)\in \AbBr$ with $|A|=p^d$ is isomorphic to $(F_{n,p^d}/I)^*$ for some ideal $I$. As every automorphism of $F_n$ maps $(p^dx)$ to itself, we immediately obtain the following variations of \cref{lem: automorphisms of F_n} and \cref{cor: lifting of isomorphisms}.
\begin{lemma}\label{lem: automorphisms Fnpd}
    For every $y\in x+F_n^2$, there exists a unique automorphism of $F_{n,p^d}^*$ mapping $x$ to $y$. 
\end{lemma}

\begin{corollary}\label{cor: lifting of isomorphisms Fnpd}
    Let $I$ and $J$ be ideals of $F_{n,p^d}$ and $\phi:(F_{n,p^d}/I)^*\to (F_{n,p^d}/J)^*$ an isomorphism, then there exists an automorphism $\hat{\phi}:F_{n,p^d}^*\to F_{n,p^d}^*$ such that $\hat{\phi}(I)= J$ and $\hat{\phi}$ is a lifting of $\phi$ in the sense that the following diagram commutes.
    \[
\begin{tikzcd}
F_{n,p^d} \arrow[d] \arrow[r, "\hat{\phi}"] & F_{n,p^d} \arrow[d] \\
F_{n,p^d}/I \arrow[r, "\phi"]               & F_{n,p^d}/J        
\end{tikzcd}\]
\end{corollary}

Consider the action of $\Aut(F_{n,p^d}^*)$ on ideals of $F_{n,p^d}$ where $\phi\in \Aut(F_{n,p^d}^*)$ maps an ideal $I$ to $\phi(I)$. We then obtain the following variation of \cref{prop: orbits and isomorphism classes general}.
\begin{theorem}\label{theorem: isomorphism classes and orbits}
    There is a bijective correspondence between isomorphism classes in $\AbBr(p^{d_1},...,p^{d_n})$ and orbits of ideals $I$ of $F_{n,p^d}$ such that $F_{n,p^d}$ is of type $(p^{d_1},...,p^{d_n})$, under the action of $\Aut(F_{n,p^d}^*)$. Under this correspondence, the orbit of an ideal $I$ is mapped to the isomorphism class of $(F_{n,p^d}/I)^*$.
\end{theorem} 

Using the bijective correspondence from \cref{prop: matrix ideal}, we can obtain an action $\cdot$ of $\Aut(F_{n,p^d}^*)$ on $\M{(p^{d_1},...,p^{d_n})}$, where $\phi\cdot M$ is the unique matrix in $\M{(p^{d_1},...,p^{d_n})}$ associated to $\phi(I_M)$, for $\phi\in \Aut(F_{n,p^d}^*)$ and $M\in \M{(p^{d_1},...,p^{d_n})}$. The orbit of some $M\in \M{(p^{d_1},...,p^{d_n})}$ under this action is denoted by $\cO(M)$. We denote the number of isomorphism classes in 
$\AbSol(p^{d_1},...,p^{d_n})$ by $|\AbSol(p^{d_1},...,p^{d_n})|$. From \cref{prop: bijective correspondence pairs and solutions} and \cref{theorem: isomorphism classes and orbits}, we obtain the main result of this section.
\begin{theorem}\label{theorem: main theorem D and orbits M}
    Isomorphism classes in $\AbSol(p^{d_1},...,p^{d_n})$ are in bijective correspondence with orbits of the action of $\Aut(F_{n,p^d}^*)$ on $\M{(p^{d_1},...,p^{d_n})}$. In particular, 
    \begin{equation*}
        |\AbSol(p^{d_1},...,p^{d_n})|=\sum_{M\in \M{(p^{d_1},...,p^{d_n})}}\frac{1}{|\cO(M)|},
    \end{equation*}
\end{theorem}

\subsection{Explicit calculations}
In what follows we will explicitly apply the preceding results of this section to the cases $n=2$ and $n=3$. 
We first cover the case where $n=2$, which was already done in \cite{JPZ2021} using different techniques, and we subsequently also give an explicit formula for $|\AbSol(p^{d_1},p^{d_2},p^{d_3})|$. For $n>3$, the applied techniques do not seem to generalise, however it would be interesting to see up to which multipermutation level $n$ and size $p^d$ one could enumerate the number of isomorphism classes of solutions with the help of a computer.
 \subsubsection{The case $n=2$}
Let $M\in \M{(p^{d_1},p^{d_2})}$. So $M$ is of the form
$$M=\begin{pmatrix}
    p^{d_1}& m \\
    0& p^{d_2},
    \end{pmatrix}$$ 
    where $m$ can be freely chosen such that $0\leq m<p^{d_2}$. 
    Recall that the automorphisms of $F_{2,p^d}^*$ are in bijection with elements in the coset $x+F_{2,p^d}$; to every element $y$ in this coset we associate the unique automorphism $\phi_{y}$ which maps $x$ to $y$. One easily checks that $\phi_{x+ax^2}\circ \phi_{x+bx^2}=\phi_{x+(a+b)x^2}$ for any $a,b\in \Z/p^d$, so $\Aut(F_{2,p^2}^*)\cong C_{p^d}$. If we let $\phi_{x+ax^2}$ act on $M$, we find that the result $\phi_{x+ax^2}\cdot M$ is the unique matrix in $\M{(p^{d_1},p^{d_2})}$ which is row-equivalent to 
    $$M=\begin{pmatrix}
    p^{d_1}& m+ap^{d_1} \\
    0& p^{d_2},
    \end{pmatrix}.$$
    As $d_1\geq d_2$, we thus find that $\phi_{x+ax^2}\cdot M$ is row-equivalent to $ M$ and therefore the action is trivial.

    \begin{theorem}\label{theorem: classification mpl 2}
        Let $p$ be a prime, $d_1\geq d_2\geq 0$ and $0\leq m<p^{d_2}$, let $I$ be the subgroup of $\Z^2$ generated by $(p^{d_1},m)$ and $(0,p^{d_2})$, and $A=\Z^2/I$. Then $$\sigma_{a+I}(b+I)=(b_1+1,a_1+b_1+b_2),$$ 
        yields a solution $(A,r)\in \AbSol(p^{d_1},p^{d_2})$.
        
        Moreover, every solution in $\AbSol(p^{d_1},p^{d_2})$ is isomorphic to such a solution for a unique choice of $d_1,d_2$ and $m$.
        \end{theorem}
        \begin{proof}
            The first part of the statement is a special case of \cref{theorem: every solution in D is of the form...} and \cref{rem: explicit form of solution}. The second part follows from \cref{theorem: main theorem D and orbits M} and the fact that the action on $\M(p^{d_1},p^{d_2})$ is trivial.
        \end{proof}
        \begin{corollary}
            Let $p$ be a prime, $d\geq 0$. The total number of solutions of size $p^d$ and multipermutation level at most 2 in $\AbSol$ is given by
    \begin{equation*}
        (p^{\lfloor d/2\rfloor+1}-1)/(p-1),
    \end{equation*}
    where $\lfloor d/2\rfloor$ is the largest integer $n$ such that $n\leq d/2$.
        \end{corollary}
        \begin{proof}
        This follow from the following calculation.
            \begin{equation*}
        \sum_{\substack{d_1+d_2=d\\d_1\geq d_2\geq 0}}|\AbSol(p^{d_1},p^{d_2})|=\sum_{\substack{d_1+d_2=d\\d_1\geq d_2\geq 0}}p^{d_2}=\sum_{d_2=0}^{\lfloor d/2\rfloor}p^{d_2}=(p^{\lfloor d/2\rfloor+1}-1)/(p-1).
    \end{equation*}
        \end{proof}
        \begin{remark}
            The classification in \cref{theorem: classification mpl 2} was also obtained in \cite{JPZ2021}, albeit in another form. It is interesting to note that the techniques used in the classification in \cite{JPZ2021} differ strongly from ours. One benefit of our techniques is that the general theoretical framework earlier in the section does not in any way pose any strict assumptions on the permutation level, the downside is that it is not immediately clear what the multiplicative and additive group of the permutation brace look like precisely. The techniques used in \cite{JPZ2021} rely strongly on the assumption that the multipermutation level is 2 and do not seem to generalise to higher multipermutation levels, however, through their classification it is immediately clear what the permutation group of a solution looks like.
        \end{remark}
    \subsubsection{The case $n=3$}
    Now let us compute the number of solutions in $\AbSol$ of a given size and multipermutation level at most 3. To do so, we first compute $|\AbSol(p^{d_1},p^{d_2},p^{d_3})|$. We may assume that $d_3> 0$, as otherwise the value can be obtained from the case $n= 2$. It will be convenient to consider the matrices in $\M{(p^{d_1},p^{d_2},p^{d_3})}$ up to row-equivalence (which we will denote by $\sim$). This is no problem as no two matrices in $\M{(p^{d_1},p^{d_2},p^{d_3})}$ are row-equivalent and the ideal $I_M$ can easily be constructed from any matrix which is row-equivalent to $M$, hence the action on $\M{(p^{d_1},p^{d_2},p^{d_3})}$ is still straightforward to compute. One easily verifies that all $M\in \M{(p^{d_1},p^{d_2},p^{d_3})}$ are row-equivalent to a unique matrix of the form
\begin{equation*}
    \begin{pmatrix}
    p^{d_1}& p^{d_1-d_2}m_2+\alpha p^{d_3} & m_1\\
    0& p^{d_2}& m_2\\
    0&0&p^{d_3}
    \end{pmatrix},
\end{equation*}
for $0\leq m_1,m_2<p^{d_3}$ and $0\leq \alpha <p^{d_2-d_3}$

Once again, for $y\in x+F_{3,p^d}^2$ we denote the unique automorphism of $F^*_{3,p^d}$ mapping $x$ to $y$ by $\phi_y$. We need to determine all $y\in x+F_{3,p^d}^2$ such that $\phi_y$ acts trivially on $M$. Let $y=x+ax^2+bx^3$. So $y^2=x^2+2ax^3$ and $y^3=x^3$ in $R$. We then know that
\begin{align*}
\phi_y\cdot M&\sim\begin{pmatrix}
    p^{d_1}& p^{d_1-d_2}m_2+\alpha p^{d_3}+ap^{d_1} & m_1+bp^{d_1}+2a(p^{d_1-d_2}m_2+\alpha p^{d_3})\\
    0& p^{d_2}& m_2+p^{d_2}\\
    0&0&p^{d_3}
    \end{pmatrix}\\
    &\sim\begin{pmatrix}
    p^{d_1}& p^{d_1-d_2}m_2+\alpha p^{d_3} & m_1+ap^{d_1-d_2}m_2\\
    0& p^{d_2}& m_2\\
    0&0&p^{d_3}
    \end{pmatrix}.
\end{align*} 
From which we find that $\phi_y\cdot M\sim M$ if and only if
\begin{equation}\label{eq: automorphism mpl 3}
   ap^{d_1-d_2}m_2\equiv 0\pmod{p^{d_3}}
\end{equation}

If we let $r(M)\in \N$ be the smallest value such that \eqref{eq: automorphism mpl 3} holds for $a=p^{r(M)}$, then we find that 
$$\Stab_{\Aut(R^*)}(M)=\{\phi_{x+ax^2+bx^3}\mid a\in p^{r(M)}\Z/p^d,b\in \Z/p^d\},$$
so in particular $|\Stab_{\Aut(R^*)}(M)|=p^{2d-r(M)}$ hence $|\cO(M)|=p^{r(M)}$.

Let $\nu_p:\Q\to \Z\cup \{+\infty\}$ denote the $p$-valuation. Then it is easily seen that $r(M)=\max\{0,-d_1+d_2+d_3-\nu(m_2)\}$. \cref{theorem: main theorem D and orbits M} now yields
\begin{align*}
    |\AbSol(p^{d_1},p^{d_2},p^{d_3})|&=\sum_{M\in \M{(p^{d_1},p^{d_2},p^{d_3})}} p^{-r(M)}\\
    &=\sum_{\substack{0\leq m_1,m_2<p^{d_3}\\0\leq \alpha <p^{d_2-d_3}}}p^{-\max\{0,-d_1+d_2+d_3-\nu(m_2)\}}\\
    &=p^{d_2}\sum_{m_2=0}^{p^{d_3}-1}p^{-\max\{0,-d_1+d_2+d_3-\nu(m_2)\}}
\end{align*}
One easily sees that for a given $v$, with $0\leq v<d_3$, there are precisely $p^{d_3-v}-p^{d_3-v-1}$ integers $k$, with $1\leq k< p^{d_3}$, such that $\nu_p(k)=v$. We therefore find that 
\begin{align*}
    |\AbSol(p^{d_1},p^{d_2},p^{d_3})|&=p^{d_2}\left(1+\sum_{v=0}^{d_3-1}\frac{p^{d_3-v}-p^{d_3-v-1}}{p^{\max\{0,-d_1+d_2+d_3-v\}}}\right)\\
    &=p^{d_2}\left(1+\sum_{w=0}^{d_3-1}\frac{p^{w+1}-p^{w}}{p^{\max\{0,-d_1+d_2+w+1\}}}\right).\\
\end{align*}
If $d_1=d_2$ then
\begin{align*}
    |\AbSol(p^{d_1},p^{d_2},p^{d_3})|&=p^{d_2}\left(1+\sum_{w=0}^{d_3-1}\frac{p^{w+1}-p^{w}}{p^{w+1}}\right)\\
    &=p^{d_2}(1+d_3(1-p^{-1})).
\end{align*}
If $d_1>d_2$ and $d_1<d_2+d_3$ then
\begin{align*}
    |\AbSol(p^{d_1},p^{d_2},p^{d_3})|&=p^{d_2}\left(1+\sum_{w=0}^{d_1-d_2-1}(p^{w+1}-p^{w})+\sum_{w=d_1-d_2}^{d_3-1}\frac{p^{w+1}-p^{w}}{p^{-d_1+d_2+w+1}}\right)\\
    &=p^{d_2}\left(1+p^{d_1-d_2}-1+p^{d_1-d_2}\sum_{w=d_1-d_2}^{d_3-1}(1-p^{-1})\right)\\
    &=p^{d_1}(1+(-d_1+d_2+d_3)(1-p^{-1})).
\end{align*}
If $d_1>d_2$ and $d_1\geq d_2+d_3$ then
\begin{align*}
    |\AbSol(p^{d_1},p^{d_2},p^{d_3})|&=p^{d_2}\left(1+\sum_{w=0}^{d_3-1}p^{w+1}-p^{w}\right)\\
    &=p^{d_2}\left(1+p^{d_3}-1\right)\\
    &=p^{d_2+d_3}
\end{align*}
We obtain the following result, whose proof is given above for $d_3>0$ and follows from \cref{theorem: classification mpl 2} for $d_3=0$.
\begin{theorem}
Let $p$ be a prime and $d_1\geq d_2\geq d_3\geq 0$, then
\begin{equation}
    |\AbSol(p^{d_1},p^{d_2},p^{d_3})|=\begin{cases}
    p^{d_1}(1+(-d_1+d_2+d_3)(1-p^{-1}))& \text{if $d_1<d_2+d_3$}\\
    p^{d_2+d_3} & \text{if $d_1\geq d_2+d_3$}
    \end{cases}.
\end{equation}
In particular, the number of isomorphism classes of solutions of size $p^d$ and multipermutation level at most 3 in $\AbSol$ can be computed as
\begin{equation*}
    \sum_{\substack{d_1\geq d_2\geq d_3\geq 0\\d_1+d_2+d_3=d}}|\AbSol(p^{d_1},p^{d_2},p^{d_3})|.
\end{equation*}
\end{theorem}

\section{Some infinite indecomposable involutive solutions with abelian permutation group}\label{section: infinite solutions}

In \cite[Theorem 6.1]{JPZ2021} for each $m\in \N$ an indecomposable involutive solutions of multipermutation level $2$ and permutation group $\Z\times \Z/m$ was given. In this section, we will show that they are the only ones. Furthermore we classify indecomposable involutive multipermutation solutions with a torsion-free abelian permutation group.

\begin{proposition}\label{prop: infinite brace mpl 2}
    Let $A$ be an infinite one-generated brace with abelian multiplicative group and multipermutation level 2, and let $X$ be a transitive cycle base of $A$. Then $(A,X)$ is isomorphic to $(F_2/I_m)^*$ for some $m\in \N\setminus \{1\}$ where $I_m=\Z m x^2$. Moreover, $(F_2/I_m,\circ)$ is isomorphic to $\Z\times (\Z/m)$.
\end{proposition}
\begin{proof}
By \cref{cor: pair isomorphic to quotient of F_n} we know that $(A,X)$ is isomorphic to $(F_2/I)^*$ for some ideal $I$. It is possible that $I=0$, in which case $I=I_0$. Now assume that $I\neq 0$. We know that $(I,+)$ should be cyclic because otherwise $(F_2/I,+)$ is finite. Let $l x+m x^2$, with $l,m\in \Z$ and $m\geq 0$, be a generator of $(I,+)$. Then $x*(l x+mx^2)=\alpha x^2\in I$ is an additive multiple of $lx+mx^2$, hence $l=0$. We therefore find that $I=I_m$ for some $m\in \N\setminus \{1\}$, as $m=1$ would mean that $\mpl(F_n/I)=1$. Because $F_2*F_2^2=F^2_2*F_2=0$, we find that $I_m$ is an ideal for every choice of $m$. As $(F_2/I_m,+)\cong \Z\times (\Z/m)$, it is clear that different choices of $m$ yield non-isomorphic braces $F_2/I_m$.

It remains to show that $(F_2/I_m,\circ)\cong \Z\times (\Z/m)$. To see this, it suffices to note that $((F_2/I_m)/(F_2/I_m)^2,\circ)\cong \Z$ and $((F_2/I_m)^2,\circ)\cong \Z/m$.
\end{proof}
\begin{theorem}\label{unique111}
    Let $m\in \N\setminus \{1\}$ then $\Z\times \Z/m$ with $$\sigma_{(a_1,a_2)}(b_1,b_2)=(b_1+1,a_1+b_1+b_2),$$ is an infinite indecomposable involutive solution with permutation group isomorphic to $\Z\times \Z/m$ and multipermutation level 2. Moreover, every infinite indecomposable solution with abelian permutation group and multipermutation level 2 is isomorphic to such a solution.
\end{theorem}
\begin{proof}
 Let $m\in \N\setminus \{1\}$. An easy computation shows that the solution associated to $(F_n/I_m)^*$ is precisely the solution in the statement and it satisfies the required properties. Conversely, if $(X,r)$ is an indecomposable solution with abelian permutation group and multipermutation level 2, then its permutation brace must satisfy the conditions of \cref{prop: infinite brace mpl 2}, hence the statement follows from \cref{prop: bijective correspondence pairs and solutions}.
\end{proof}

\begin{proposition}\label{prop: brace torsion free mult group}
Let $A$ be a brace having multipermutation level $n$, with a transitive cycle base $X$ and with an abelian torsion-free multiplicative group. Then $(A,X)$ is isomorphic to $F_{n}^*$.
\end{proposition}
\begin{proof}
    From \cref{cor: pair isomorphic to quotient of F_n} we know that $A\cong F_n/I$ for some ideal $I$. Assume that $I\neq 0$. As $\Soc(F_n)=F_n^n$, it follows that $I\cap F_n^n\neq 0$ by \cite[Theorem 2.8]{CSV19}. If $I\cap F_n^n\neq F_n^n$, we find that $x^n$ has a finite order in $(F_n,\circ)$. This implies that $I\cap F_n^n= F_n^n$, but then $F_n/I$ has a multipermutation level strictly less than $n$. We conclude that $I=0$.
\end{proof}

\begin{theorem}\label{theorem: unique}
Let $n\in \mathbb{N}$ with $n>1$. Then the solution $(F_n,r_x)$ is the unique indecomposable involutive solution with abelian torsion-free permutation group and multipermutation level $n$.
\end{theorem}

\begin{proof}
This is a direct consequence of \cref{prop: bijective correspondence pairs and solutions} and \cref{prop: brace torsion free mult group}.
\end{proof}
\begin{remark}
    The solutions in \cref{theorem: unique} can be expressed explicitly in terms of the additive group $\Z^n$ of $F_n$, just as the ones discussed in \cref{rem: explicit form of solution}.
\end{remark}

\section*{Acknowledgment}
The authors would like to thank Wolfgang Rump for providing them with the preprint of \cite{rump2022class} and for the clarification on his comment herein on the equality $|A|=|A^2||\Soc(A)|$.

\bibliographystyle{amsalpha}

\bibliography{Bibliography.bib,bib.bib}

\end{document}